\documentclass{article}
\usepackage{fancyhdr}
\usepackage{amscd}
\usepackage[dvipdfmx]{graphicx}
\usepackage{amsmath}
\usepackage{amssymb}
\usepackage{amsthm}
\usepackage{mathrsfs}
\newtheorem{dfn}{Definition}[section]
\newtheorem{thm}[dfn]{Theorem}
\newtheorem{prop}[dfn]{Proposition}

\numberwithin{equation}{section}

\begin{document}

\title{
Dynamics and structure of pull-in and touchdown behavior in parallel-plate electrostatic MEMS actuators via geometric approaches
}

\author{Yu Ichida
\thanks{Department of Mathematics, School of Science and Technology, Meiji University, 1-1-1 Higashimita, Tama-ku, Kawasaki 214-8571, Japan, {\tt ichidayu@meiji.ac.jp}}, 
Kyoichi Kakuno
\thanks{Graduate School of Science and Engineering, Ritsumeikan University, Kusatsu 525-8577, Japan
{\tt rm0199rk@ed.ritsumei.ac.jp}},
Daisuke Yamane
\thanks{Graduate School of Science and Engineering, Ritsumeikan University, Kusatsu 525-8577, Japan
{\tt dyamane@fc.ritsumei.ac.jp}}
$^{,}$
\footnote{Ritsumeikan Advanced Research Academy, Kyoto, 604-8520, Japan}
}
%\date{}
\maketitle

\begin{abstract}
This paper focuses on parallel-plate electrostatic actuators, which are  fundamental structures found in Micro-Electro-Mechanical Systems (MEMS) for numerous modern devices.
It considers the equations used to model these actuators.
In particular, we investigate the behavior of the solutions to a MEMS model that is represented by a second-order ordinary differential equation, where the spring is characterized as soft, linear, or hard.
We present the mathematical structure underlying the pull-in and touchdown phenomena that characterize the behavior of systems incorporating soft and hard spring, based on the behavior of linear springs.
The pull-in phenomenon corresponds to the degeneration of the solutions to the model equation. 
The touchdown phenomenon corresponds to the finite-time singularity of the solutions.
These global dynamics and structures are systematically revealed through geometric approaches based on Poincar\'e-type compactification, the center manifold theorem, and blow-up technique. 
These methods successfully induce and resolve the dynamics at infinity.\end{abstract} 


\section{Introduction}
\label{sec:IKYsh-int}
Small, high-performance devices and micromachines are indispensable to today's digital society. 
Micro-electro-mechanical systems (MEMS), which integrate electrical and mechanical elements, are built into many modern devices.
For an overview of MEMS and their model equations, see \cite{Sen} and the references therein.
MEMS are used in sensors and switches, as well as in fields such as medicine and biochemistry.

Investigating the effects of changes in material properties is essential for MEMS development. 
However, conducting these types of investigations requires significant time, cost, and resources.
Therefore, formulating hypotheses and logical arguments using mathematical methods is valuable.

Our group studies the impact of different material properties and plate geometries, beginning with parallel-plate electrostatic actuators. 
These actuators are the simplest form of MEMS microstructures.
In this paper, we first study the influence of spring characteristics from an engineering perspective.
Specifically, we compare hard, soft, and linear springs to clarify how device behavior differs.

\begin{figure}[t]
\begin{center}
\includegraphics[scale=0.65]{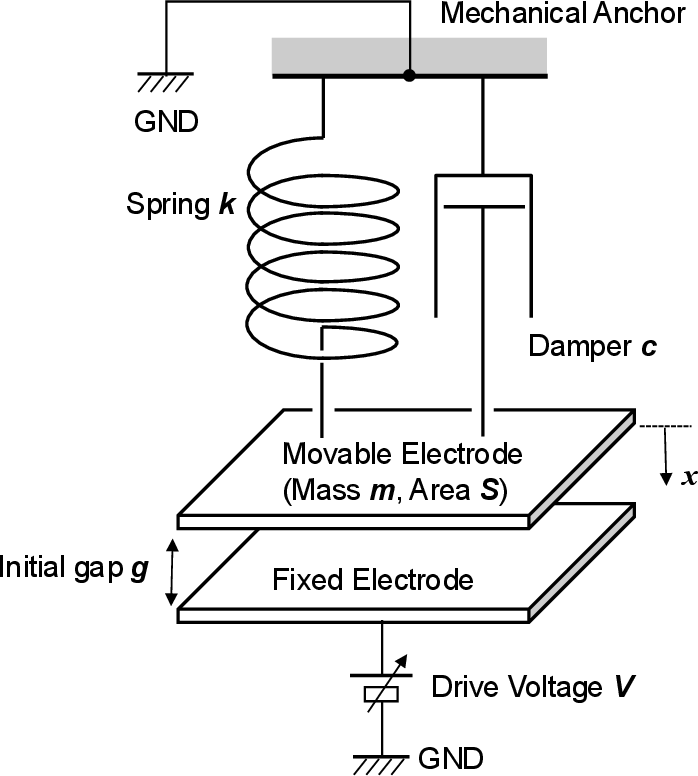}
\caption{Schematic picture of a parallel-plate electrostatic actuator.}
\label{fig:IKYsh-int1}
\end{center}
\end{figure}

A parallel-plate electrostatic actuator describes the motion of two parallel plates, as shown in the schematic picture in Figure \ref{fig:IKYsh-int1}.
This model is also known as the spring-mass model.
Let $x=x(t)$ represent the distance between the plates, where $0\le x(t) \le g$, and let $g$ represent the initial gap between them.
Let us derive the model equation for a parallel-plate electrostatic actuator.
From Newton's laws of motion, we can formulate the equation of motion as 
\begin{equation}
m\dfrac{d^{2}x}{dt^{2}}=\sum\, {\rm{forces}}.
\label{eq:IKYsh-int1}
\end{equation}
Here, $m$ represents the mass.
Let b represent the spring characteristics. 
Based on the spring force, velocity-proportional damping, and electrostatic attraction, we can derive the following:
\begin{equation}
\sum\, {\rm{forces}}=-kx -bx^{3}-c\dfrac{dx}{dt} + \dfrac{1}{2}\dfrac{\varepsilon_{0}S}{(g-x)^{2}}V^{2}.
\label{eq:IKYsh-int2}
\end{equation}
For $b \in \mathbb{R}$, a spring is considered a hard spring when $b > 0$, a linear spring when $b = 0$, and a soft spring when $b < 0$.
Furthermore, the spring characteristics are modeled based on the Duffing equation term $bx^3$, which is based on empirical rules from the field of engineering (for instance, see \cite{Sen}).

$c$ is the velocity parameter, $k$ is the spring constant, $\varepsilon_{0}$ is the permittivity, $S$ is the plate area, and $g$ is the initial gap.
All of these are positive constants.
In this paper, we adopt a constant voltage $V(t) \equiv V > 0$, which is also commonly used in MEMS experiments for simplicity.

As a contribution to the field of engineering, understanding the behavior of solutions through mathematical analysis is expected to provide insights into the effects of spring characteristics and lead to improvements in MEMS functionality.
By combining the model equations \eqref{eq:IKYsh-int1} and \eqref{eq:IKYsh-int2}, we have
\begin{equation}
m\dfrac{d^{2}x}{dt^{2}}+c\dfrac{dx}{dt}+kx+bx^{3}=\dfrac{1}{2}\dfrac{\varepsilon_{0}S}{(g-x)^{2}}V^{2}.
\label{eq:IKYsh-int3}
\end{equation}

We impose the initial conditions $x(0)=0$ and $\dot{x}(0)=0$.
Hereinafter, we define $\dot{x}=dx/dt$ and $\ddot{x}=d^{2}x/dt^{2}$.
The mathematical model discussed in this paper is an initial value problem with the governing equation given by \eqref{eq:IKYsh-int3}.

In parallel-plate electrostatic actuators, the pull-in, touchdown, and the 1/3 rule are well-known phenomena in the MEMS field when $b=0$.
These behaviors arise from the changing relationship between electrostatic attraction and spring restoring force when the magnitude of the applied voltage is varied.
In this paper, the voltage at which electrostatic attraction and spring restoring force are in a state of equality is referred to as the pull-in voltage.
For an initial gap $g$, it is known that pull-in occurs at $x = g/3$, a phenomenon known as the 1/3 rule.
For instance, see \cite{Nath, ZYPM} and the references therein.

Increasing the voltage even slightly beyond the pull-in voltage causes the electrostatic force to exceed the spring's restoring force, resulting in instantaneous contact between the two plates.
We refer to this as the touchdown phenomenon in this paper.
While this phenomenon is valuable for switch applications, it must be avoided in many MEMS devices to ensure safety.

To refine the 1/3 rule and clarify the contact velocity, researchers are experimenting with MEMS development by altering the shape of parallel plates or modifying material properties. 
From a mathematical perspective, a theoretical understanding of these three characteristics is expected to reduce implementation costs, improve efficiency, enhance performance, and encourage technological innovation.

From a mathematical viewpoint, the first mathematical model of a parallel-plate electrostatic actuator is generally attributed to Nathanson et al. (\cite{Nath}).
Since then, a great deal of research has been conducted, including studies focused on mathematical topics, such as proving the existence of periodic solutions to differential equations.
See Ai-Pelesko (\cite{AiPel}), Guti\'errez-Torres (\cite{GT}), Kong-Yu (\cite{KoYu}), Cheng-Zhao-Yuan (\cite{CZY26}), and references therein.
In these studies, the applied voltage $V(t)$ is treated as a time-dependent periodic voltage, and the properties of periodic solution behavior (existence, stability, and multiplicity) as well as the bifurcation structure are investigated.

Many mathematical studies focus on rigorously proving that the behavior of two parallel plates becomes periodic when the applied voltage is defined as a function, such as $V = V(t)$.
However, few studies compare these results with experimental data.
Furthermore, the equation contains nonlinear terms involving negative powers, making its analysis challenging.

Furthermore, regarding the touchdown phenomenon, the authors believe that while much research has been conducted using partial differential equations as model equations, there has been little research that integrates MEMS experiments from the perspectives of pull-in and the 1/3 rule.

This paper aims to rigorously elucidate three characteristic changes in behavior caused by spring properties that have not yet been clearly demonstrated from an engineering perspective.
Specifically, we intend to shed new light on MEMS development by clarifying the following two points:
\begin{enumerate}
\item[(i)] 
classification of the behavior of \eqref{eq:IKYsh-int3} under the assumptions that $b = 0$ and the voltage is constant, as well as the mathematical structure of the three characteristic behaviors (pull-in, touchdown, and the 1/3 rule).
\item[(ii)] 
the influence of $b$ on the mathematical structure underlying the three characteristic behavior resulting from the consideration of spring characteristics ($b>0$ and $b<0$).
\end{enumerate}
In the experimental approach, we fabricate a demonstration unit equivalent to a parallel-plate electrostatic actuator model and extract three characteristic changes in behavior by conducting a statistically sufficient number of experiments using linear and nonlinear springs.
However, since it is difficult to directly and accurately observe the movement speed and displacement of the moving parts of micro- or nano-scale parallel plates, it is generally difficult to demonstrate the effects of spring characteristics using a prototype.

We will discuss the strategy for deriving these results.
This will lay the groundwork for a rigorous mathematical investigation into the behavior of solutions to the differential equation \eqref{eq:IKYsh-int3}.
By formulating this as an initial value problem, we can impose the initial condition
\begin{equation}
x(0)=0,\quad \dot{x}(0)=0.
\label{eq:IKYsh-int4}
\end{equation}
Hereinafter, let $dx/dt = \dot{x}$ and $d^{2}x/dt^{2} = \ddot{x}$.
The term $1/(g-x)^{2}$ on the right-hand side of this equation means that the solution $x(t)$ cannot be expressed in terms of known functions. 
Therefore, we will proceed with a qualitative analysis of the solution. 
However, the presence of the negative power term complicates this analysis.
Therefore, we introduce a new variable denoted by the symbol $\phi(t)$ as follows:
\begin{equation}
\phi(t)=g-x(t).
\label{eq:IKYsh-int5}
\end{equation}
Thus, the equation that the function $\phi$ must satisfy and the corresponding initial conditions are
\begin{equation}
-m\ddot{\phi}-c\dot{\phi}+k(g-\phi)+b(g-\phi)^{3}=\dfrac{1}{2}\varepsilon_{0}SV^{2}\phi^{-2}
\label{eq:IKYsh-int6}
\end{equation}
and
\begin{equation}
\phi(0)=g,\quad \dot{\phi}(0)=0.
\label{eq:IKYsh-int7}
\end{equation}
Furthermore, by setting $\dot{\phi} = \psi$, $\eqref{eq:IKYsh-int6}$ is equivalent to the following:
\begin{equation}
\begin{cases}
\dot{\phi}=\psi,\\
\dot{\psi}=-\dfrac{c}{m}\psi -\dfrac{b}{m}\phi^{3}+\dfrac{3gb}{m}\phi^{2} - \dfrac{3g^{2}b+k}{m}\phi + \dfrac{kg+bg^{3}}{m} - \dfrac{\varepsilon_{0}SV^{2}}{2m}\phi^{-2}.
\end{cases}
\label{eq:IKYsh-int8}
\end{equation}
Note that nondimensionalization is usually used to reduce the number of parameters, such as $c$, $m$, $k$, $b$, $g$, $\varepsilon_{0}$, $S$ and $V$, in mathematical treatments. 
Unlike standard mathematical treatments that rely on nondimensionalization to reduce parameter spaces, this paper intentionally preserves all physical parameters. 
This approach ensures that the bifurcation boundaries and invariant manifolds unraveled by our geometric analysis remain transparently linked to real-world design spaces, directly bridging abstract singularity theory with experimental prototype specifications.
By avoiding nondimensionalization, the explicit relationship between the mathematical structure of singularities and the physical specifications of the device prototypes remains completely transparent, thereby facilitating direct comparison with experimental data.

A touchdown is defined as the function $x(t)$ reaching $g$ within a finite time (see Definition \ref{def:IKYsh-mr1} below for a detailed definition), and this corresponds to the behavior of $\psi\to -\infty$ in \eqref{eq:IKYsh-int8}.
In general, this behavior corresponds to the finite-time singularity of the solution, and its analysis is not straightforward.
To analyze the behavior of $\psi\to -\infty$, this paper introduces the Poincar\'e-type compactification, a type of compactification of phase space.
This is a method for deriving dynamical systems on $\Phi=\mathbb{R}^{2} \, \cup\, \{\|(\phi, \psi)\|=+\infty\}$ as defined in \eqref{eq:IKYsh-int8}. 
It represents a fusion of dynamical systems theory and geometry.
For details on Poincar\'e-type compactifications, see \cite{FAL, QTW, UPKPP, Matsue1, Matsue2}, or the brief review in Subsection \ref{sub:IKYsh-d1}.
Since the leading terms differ for $b=0$ and $b \neq 0$, the treatment of infinity changes accordingly.
Therefore, in this paper, we will analyze the dynamical systems at infinity separately for $b=0$ and $b \neq 0$.

By considering the dynamical system near the finite equilibrium for \eqref{eq:IKYsh-int8}, we can obtain the mathematical structure underlying the pull-in and the 1/3 rule.
The moment when the existence or stability of a local (finite) equilibrium changes corresponds to a pull-in, which mathematically corresponds to a bifurcation point.
Furthermore, by examining the dynamical system at infinity reveals the mathematical structure underlying the touchdown.
A touchdown is a finite-time singularity of a solution to a differential equation. 

Since analyzing these behavior amounts to analyzing degeneracy and infinity, it is no easy work to treat this central problem in the dynamical systems theory.
To overcome these difficulties, we employ geometric approaches based on a Poincar\'e-type compactification, the center manifold theorem, and blow-up technique. 
Geometric treatments of finite-time singularities for autonomous ordinary differential equations (ODEs) have been framework established in recent literature (e.g., \cite{Matsue1, Matsue2}) by using the Poincar\'e-type compactification.
A comprehensive geometric classification and understanding the behavior that fully captures the complex interaction between the electrostatic forces and spring nonlinearities ($b \in \mathbb{R}$) remains unsolved.
The previous studies have analyzed touchdown instabilities within specific partial differential equations (PDEs) or localized ODE frameworks, a comprehensive geometric classification that captures the interplay between spring nonlinearities ($b \in \mathbb{R}$) and boundaries at infinity has remained unexplored. 

This research arises from concerns about the practical development of MEMS and their associated mathematical contributions.
However, since the analysis of model equations exhibits mathematically intriguing properties, such as complex interaction, degeneracy and finite-time singularity.
Thus, our contributions extend beyond MEMS to the field of mathematics.
The core novelty of this work lies in fully characterizing the topology of the phase space and completely classifying behavior.
To the best of the authors' knowledge, this is the first study to rigorously clarify the exact mathematical relationship between stable operations, the pull-in bifurcation point, and touchdown singularities while preserving the full physical parameters of the electrostatic MEMS actuator.

To the authors' knowledge, this is the first study to provide a rigorous, unified mathematical framework that classifies how the soft ($b < 0$), linear ($b = 0$), and hard ($b > 0$) spring characteristics alter the global phase space structure, stable equilibria, and the nature of touchdown singularities.

The organization of this paper is as follows.
The next section presents the main results.
Section \ref{sec:IKYsh-f} discusses the finite equilibrium of \eqref{eq:IKYsh-int8}.
Sections \ref{sec:IKYsh-d} and \ref{sec:IKYsh-dib} provide classifications of the dynamical system \eqref{eq:IKYsh-int8} at infinity for the cases $b=0$ and $b \neq 0$, respectively.
This dynamical system, which includes infinity, is also known as a Poincar\'e-type disk and is obtained via a Poincar\'e-type compactification.
In Section \ref{sec:IKYsh-pro}, we provide a proof of the main results.
Finally, Section \ref{sec:IKYsh-c} provides insights and considerations for improving the structure and functionality of MEMS.

\section{Main results}
\label{sec:IKYsh-mr}
Before presenting the main results, we will define the solution to the model equation that describes the touchdown.
Mathematically, this corresponds to a quench, which is a type of finite-time singularity.

\begin{dfn}
\label{def:IKYsh-mr1}
If there exists a time $|t_{+}| < +\infty$ such that
\[
\lim_{t\to t_{+}-0}x(t) =g, \quad 
\lim_{t\to t_{+}-0} \dot{x}(t)= +\infty,
\]
then the solution to the equation \eqref{eq:IKYsh-int3} is said to touchdown (or quench) in finite time, and $x = x(t)$ is called a touchdown solution.
\end{dfn}

\begin{thm}
\label{thm:IKYsh-mr1}
Let $c$, $m$, $k$, $g$, $\varepsilon_{0}$, $S$, and $V$ be positive constants, and let $b = 0$.
Let $K$ be defined as follows:
\[
K= (2m)^{-1}\varepsilon_{0}SV^{2} - (27m)^{-1}4kg^{3} 
\]
Then, depending on the value of $K$, the solutions to the equation \eqref{eq:IKYsh-int3} under the initial conditions $x(0) = 0$ and $\dot{x}(0) = 0$ can be classified into the following three types of behavior:
\begin{enumerate}
\item[(i)] 
If $K > 0$, the solution of \eqref{eq:IKYsh-int3} is a touchdown solution (Definition \ref{def:IKYsh-mr1}).
In addition, the asymptotic behavior of $x(t)$ as $t\to t_{+}-0$ is 
\begin{equation}
x(t) \sim g - A(t_{+}-t)^{\frac{2}{3}} \quad {\rm{as}} \quad t\to t_{+}-0.
\label{eq:IKYsh-mr1}
\end{equation} 
\item[(ii)] 
If $K=0$, then the solution $x(t)$ to this initial value problem exists globally and satisfies the following:
\[
x(t)=\dfrac{1}{3}g +O(t^{-1}) \quad {\rm{as}} \quad t\to +\infty,
\]
that is, 
\[
\lim_{t\to \infty}x(t)=\dfrac{1}{3}g, \quad 
\lim_{t\to \infty} \dot{x}(t)=0.
\]
\item[(iii)]
If $K < 0$, then the solution $x(t)$ to this initial value problem is global.
It satisfies $x(t) \to M$ ($0 < M < g/3$) and $\dot{x}(t) \to 0$ as $t\to +\infty$.
\end{enumerate}
\end{thm}

In particular, when $K=0$, the pull-in voltage $V=V^{*}$ is
\begin{equation}
V^{*}=\sqrt{\dfrac{8}{27}\dfrac{k}{\varepsilon_{0}S}g^{3}}.
\label{eq:IKYsh-mr2}
\end{equation}
The case $K=0$ corresponds to the critical pull-in state physically, and mathematically represents the threshold of a saddle-node bifurcation, where stable and unstable finite equilibria coalesce.

The following result concerns the classification of solutions to the initial value problem when incorporating a hard spring with $b>0$ or a soft spring with $b<0$. 
By comparing this result with Theorem \ref{thm:IKYsh-mr1}, it is possible to discuss the effects of replacing a linear spring with a soft or hard spring.

\begin{thm}
\label{thm:IKYsh-mr2}
Let $0\neq b\in \mathbb{R}$, $c$, $m$, $k$, $g$, $\varepsilon_{0}$, $S$ and $V$ be positive constants. 
The equation
\begin{equation}
5b\phi^{3}-12bg\phi^{2}+3(k+3bg^{2})\phi-2(kg+bg^{3})=0
\label{eq:IKYsh-mr3}
\end{equation}
has exactly one real root $\phi = \varphi^{*}$ for $0 < \phi < g$.
As $\phi = \varphi^{*}$, let $F(\varphi^{*})$ be defined by 
\begin{equation}
F(\varphi^{*}) := b(\varphi^{*})^{5}-3bg(\varphi^{*})^{4}+(k+3bg^{2})(\varphi^{*})^{3}-(kg+bg^{3})(\varphi^{*})^{2}+2^{-1}\varepsilon_{0}SV^{2}.
\label{eq:IKYsh-mr4}
\end{equation}
Then, depending on the value of $F(\varphi^{*})$, the behavior can be classified into the following three cases:
\begin{enumerate}
\item[(i)]
If $F(\varphi^{*}) > 0$, then a touchdown occurs as stated in Theorem \ref{thm:IKYsh-mr1} (i).
In addition, the asymptotic behavior of $x(t)$ as $t \to t_{+} - 0$ is given by \eqref{eq:IKYsh-mr1} .
\item[(ii)] 
If $F(\varphi^{*}) = 0$, let $x^{*} = g - \varphi^{*}$, then the solution $x(t)$ to this initial value problem is a global solution.
In addition, the following holds:
\[
x(t)=x^{*} +O(t^{-1}) \quad {\rm{as}} \quad t\to +\infty,
\]
that is,
\[
\lim_{t\to \infty}x(t)=x^{*}, \quad 
\lim_{t\to \infty} \dot{x}(t)=0.
\]
Here, $0<x^{*}<g/3$ holds for $b<0$, $x^{*}=g/3$ holds for $b=0$ and $g/3 < x^{*} <17g/27$ holds for $b>0$.
The system reaches the saddle-node bifurcation threshold (pull-in state), where the two finite equilibria coalesce into a unique degenerate finite equilibrium.
\item[(iii)] 
If $F(\varphi^{*}) < 0$, then the solution $x(t)$ to this initial value problem is a global solution, and $x(t) \to M$ and $\dot{x}(t) \to 0$ as $t \to +\infty$.
Regarding the estimation of $M$, $0 < M < 10g/27$ holds for $b > 0$ and $0 < M < g/3$ holds for $b < 0$.
\end{enumerate}
\end{thm}

It is immediately clear that if we set $b=0$ in Theorem \ref{thm:IKYsh-mr2}, then \eqref{eq:IKYsh-mr3} becomes $3k\phi - 2kg = 0$, and its solution is 
\[
\phi=\varphi^{*}=\dfrac{2}{3}g,
\]
which satisfies $0 < \phi < g$.
Furthermore, since \eqref{eq:IKYsh-mr4} yields $F\left(2g/3\right)=K$, we see that Theorem \ref{thm:IKYsh-mr2} includes Theorem \ref{thm:IKYsh-mr1}.

\section{Dynamics near finite equilibria}
\label{sec:IKYsh-f}
Since $\phi=0$ is a singularity of \eqref{eq:IKYsh-int8}, we introduce the following the time-rescaling desingularization:
\begin{equation}
dt/ds=\phi^{2}.
\label{eq:IKYsh-f1}
\end{equation}
This time-rescaling desingularization \eqref{eq:IKYsh-f1} corresponds to multiplying the vector field by $\phi^{2}$.
The solution curves of the system are invariant under both the parameters $t$ and $s$.
See \cite{CK} and the references therein for details.

We then obtain the following system:
\begin{equation}
\begin{cases}
\phi'=\phi^{2}\psi,\\
\psi'=-\dfrac{c}{m}\phi^{2}\psi -\dfrac{b}{m}\phi^{5}+\dfrac{3gb}{m}\phi^{4} - \dfrac{3g^{2}b+k}{m}\phi^{3} + \dfrac{kg+bg^{3}}{m}\phi^{2} - \dfrac{\varepsilon_{0}SV^{2}}{2m}
\end{cases}
\label{eq:IKYsh-f2}
\end{equation}
with the initial conditions $\phi(0)=g$ and $\psi(0)=0$.
Here, we define $\phi'=d\phi/ds$ and $\psi'=d\psi/ds$.
Corresponding to $0 \leq x(t) \leq g$, it is sufficient to consider the range $0 \leq \phi \leq g$.

The equations satisfied by the (finite) equilibria of \eqref{eq:IKYsh-f2} are given by
\[
\begin{cases}
\phi^{2}\psi=0, \\
-\dfrac{c}{m}\phi^{2}\psi -\dfrac{b}{m}\phi^{5}+\dfrac{3gb}{m}\phi^{4} - \dfrac{3g^{2}b+k}{m}\phi^{3} + \dfrac{kg+bg^{3}}{m}\phi^{2} - \dfrac{\varepsilon_{0}SV^{2}}{2m}=0.
\end{cases}
\]
Although the first equation implies that either $\phi=0$ or $\psi=0$, setting $\phi=0$ is not suitable since it does not satisfy the second equation.
Thus, $\psi=0$ must hold.
We then define the function $F=F(\phi)$ as follows:
\begin{equation}
F(\phi)=b\phi^{5}-3bg\phi^{4}+(k+3bg^{2})\phi^{3}-(kg+bg^{3})\phi^{2}+\dfrac{1}{2}\varepsilon_{0}SV^{2}.
\label{eq:IKYsh-f3}
\end{equation}
In other words, the problem of finding the finite equilibria of \eqref{eq:IKYsh-f2}  is equivalent to the problem of finding the zero points of $F(\phi)$.
Therefore, we will solve the former problem by examining the graph of $F(\phi)$.
Although the conclusion remains the same, the intermediate steps differ.
Since the situations vary depending on whether $b=0$, $b>0$, or $b<0$, we will consider each case separately for the reader's convenience.

\subsection{Existence and stability of the finite equilibria for $b=0$}
\label{sub:IKYsh-f0}
The function $F = F(\phi)$, defined by \eqref{eq:IKYsh-f3}, can be expressed as follows when $b = 0$:
\[
F(\phi) = k\phi^{3}-kg\phi^{2}+\dfrac{1}{2}\varepsilon_{0}SV^{2}.
\]
By differentiating $F(\phi)$ with respect to $\phi$, we can immediately see that it has extrema at $\phi = 0$ and $\phi = 2g/3$.
Since $F(0) = F(g) > 0$, the value of $F(2g/3)$ determines whether $F(\phi)$ has zeros, as well as their number and locations.
We then set 
\[
K=F(2g/3)=(2m)^{-1}\varepsilon_{0}SV^{2}-(27m)^{-1}4kg^{3}.
\]
Depending on the value of $K$, the equilibria can be classified as follows:
\begin{enumerate}
\item[(i)] 
When $K = F(2g/3) > 0$, There is no root $\phi\in(0, g)$ that satisfies $F(\phi) = 0$.
Therefore,the system \eqref{eq:IKYsh-f2} has no equilibria.
\item[(ii)] 
When $K = F(2g/3) = 0$, the equilibrium of \eqref{eq:IKYsh-f2} for $b=0$ is $E_0: (\phi, \psi) = (2g/3, 0)$.
This equilibrium corresponds to the balance between the electrostatic attraction, which follows the 1/3 rule at the pull-in voltage, and the restoring force of the spring at the initial gap of $1/3$.
Solving for $V$ at $K=0$ yields the pull-in voltage \eqref{eq:IKYsh-mr2}.
\item[(iii)] 
When $K=F\left(2g/3\right)<0$, the system \eqref{eq:IKYsh-f2} for $b=0$ has exactly two equilibria $E_{1}: (\phi,\psi)=(\phi_{1}, 0)$ and $E_{2}: (\phi_{2}, 0)$ with $0<\phi_{1}<2g/3<\phi_{2}<g$.
\end{enumerate}

Next, we will discuss the local stability of these equilibria.
The Jacobian matrix $J_{1}$ of the vector field \eqref{eq:IKYsh-f2} for $b=0$ and $K<0$ at $E_{1}$ is 
\[
J_{1}=\left(\begin{array}{cc}
0 & \phi_{1}^{2} \\
-m^{-1}3k\phi_{1}^{2}+m^{-1}2kg\phi_{1} & -m^{-1}c\phi_{1}^{2}
\end{array}\right).
\]
Hence, we can conclude that the equilibrium $E_{1}: (\phi, \psi)=(\phi_{1}, 0)$ is a saddle according to the standard dynamical systems theory.

Similarly, the Jacobian matrix $J_{2}$ of the vector field \eqref{eq:IKYsh-f2} at $E_{2}$ is 
\[
J_{2}=\left(\begin{array}{cc}
0 & \phi_{2}^{2} \\
-m^{-1}3k\phi_{2}^{2}+m^{-1}2kg\phi_{2} & -m^{-1}c\phi_{2}^{2}
\end{array}\right).
\]
A simple calculation shows that equilibrium $E_2$ is asymptotically stable.
Whether it is a node or a spiral depends on the sign of the discriminant of the characteristic polynomial.

The rest of this section is devoted to analyzing the dynamical system near the equilibrium $E_{0}: (\phi, \psi) = (2g/3, 0)$ when $K=0$.
The Jacobian matrix $J_{0}$ of the vector field \eqref{eq:IKYsh-f2} for $b=0$ at $E_{0}$ is 
\[
J_{0}=\left(\begin{array}{cc}
0 & 4g^{2}/9 \\
0 & -4cg^{2}/9m
\end{array}\right).
\]
The eigenvalues of $J_{0}$ are $0$ and $-4cg^{2}/9m<0$ and the corresponding eigenvectors are 
\[
{\mathbf{v}}_{1}=(1,0)^{T}, \quad 
{\mathbf{v}}_{2}=(9^{-1}4g^{2}, -(9m)^{-1}4cg^{2})^{T}
\]
respectively.
Here, $T$ denotes the symbol of transpose.
Since the equilibrium $E_{0}$ is not hyperbolic, we apply the center manifold theory to study the dynamics near $E_{0}$.
See \cite{carr} and references therein for details about center manifold theory. 

First, we introduce the following transformation for \eqref{eq:IKYsh-f2}:
\[
\bar{\phi}=\phi-\dfrac{2}{3}g, \quad
\bar{\psi}=\psi.
\]
Then the following system holds:
\begin{equation}
\begin{cases}
\bar{\phi}'=\bar{\phi}^{2}\bar{\psi}+3^{-1}4g\bar{\phi}\bar{\psi}+9^{-1}4g^{2}\bar{\psi},
\\
\bar{\psi}'=-m^{-1}c\bar{\phi}^{2}\bar{\psi}-(3m)^{-1}4cg\bar{\phi}\bar{\psi}-(9m)^{-1}4cg^{2}\bar{\psi} -m^{-1}k\bar{\phi}^{3}-m^{-1}kg\bar{\phi}^{2}.
\end{cases}
\label{eq:IKYsh-b01}
\end{equation}

Next, we set the matrix $P$ as $P=(\mathbf{v}_{1},\mathbf{v}_{2})$.
Then, the equation \eqref{eq:IKYsh-b01} can be written as follows:
\begin{align*}
\left(\begin{array}{cc}
\bar{\phi}' \\
\bar{\psi}'
\end{array}
\right) 
&= \left(\begin{array}{cc}
0 & 4g^{2}/9 \\
0 & -4cg^{2}/9m
\end{array}
\right)\left(\begin{array}{cc}
\bar{\phi} \\
\bar{\psi}
\end{array}
\right)+{\mathbf{a}}
\\
&= PP^{-1}\left(\begin{array}{cc}
0 & 4g^{2}/9 \\
0 & -4cg^{2}/9m
\end{array}
\right)PP^{-1}\left(\begin{array}{cc}
\bar{\phi} \\
\bar{\psi}
\end{array}
\right)+{\mathbf{a}}
\\
&= P\left(\begin{array}{cc}
0 & 0 \\
0 & -4cg^{2}/9m
\end{array}
\right)P^{-1}\left(\begin{array}{cc}
\bar{\phi} \\
\bar{\psi}
\end{array}
\right)+{\mathbf{a}},
\end{align*}
where ${\mathbf{a}}$ is defined as follows:
\[
{\mathbf{a}}=\left(\begin{array}{cc}
\bar{\phi}^{2}\bar{\psi}+3^{-1}4g\bar{\phi}\bar{\psi} \\
-m^{-1}c\bar{\phi}^{2}\bar{\psi}-(3m)^{-1}4cg\bar{\phi}\bar{\psi}-m^{-1}k\bar{\phi}^{3}-m^{-1}kg\bar{\phi}^{2}
\end{array}
\right).
\]
Let 
\[
\left(\begin{array}{cc}
\tilde{\phi} \\
\tilde{\psi}
\end{array}
\right)=T^{-1}\left(\begin{array}{cc}
\bar{\phi} \\
\bar{\psi}
\end{array}
\right).
\]
We then obtain the following system:
\begin{equation}
\begin{cases}
\tilde{\phi}'=-c^{-1}kg\tilde{\phi}^{2}-(9c)^{-1}8kg^{5}\tilde{\phi}\tilde{\psi} -(81c)^{-1}16kg^{5}\tilde{\psi}^{2}+O(3), \\
\tilde{\psi}'=-(9m)^{-1}4cg^{2}\tilde{\psi}+(4cg)^{-1}9k\tilde{\phi}^{2}+(3mc)^{-1}(6kmg^{3}-4c^{2}g)\tilde{\phi}\tilde{\psi}\\
\quad\quad\quad +(27mc)^{-1}(12g^{3}-16c^{2}g^{3})\tilde{\psi}^{2}+O(3).
\end{cases}
\label{eq:IKYsh-b02}
\end{equation}
Here, $O(3)$ denotes terms of degree three or higher in $\tilde{\phi}$ and $\tilde{\psi}$.
The center manifold theory is applicable to study the dynamics of system \eqref{eq:IKYsh-b02}.
It implies that there exists a function $h(\tilde{\phi})$ satisfying
\[
h(0)=\dfrac{dh}{d\tilde{\phi}}(0)=0
\]
such that the center manifold of the origin for \eqref{eq:IKYsh-b02} is locally represented as $\{(\tilde{\phi},\tilde{\psi}) \,|\, \tilde{\psi}(s)=h(\tilde{\phi}(s))\}$.
Differentiating it with respect to $s$, we obtain that the approximation of the (graph of) center manifold is 
\[
\left\{ (\tilde{\phi}, \tilde{\psi}) \,|\, \tilde{\psi}=\dfrac{81km}{16c^{2}g^{3}}\tilde{\phi}^{2}+O(\tilde{\phi}^{3}) \right\}.
\]
Therefore, the dynamics of \eqref{eq:IKYsh-b02} near the origin is topologically conjugate to the dynamics of the following equation:
\[
\tilde{\phi}'=-\dfrac{kg}{c}\tilde{\phi}^{2}+O(\tilde{\phi}^{3}).
\]
Based on the above discussion and using the transformation backward with
\[
\bar{\phi}=\tilde{\phi}+\dfrac{4g^{2}}{9}\tilde{\psi}, \quad
\bar{\psi}=-\dfrac{4cg^{2}}{9m}\tilde{\psi}, \quad
\phi=\bar{\phi}+\dfrac{2}{3}g, \quad
\psi=\bar{\psi}
\]
the approximation of the center manifold near $E_0$ is
\begin{equation}
\left\{ (\phi, \psi) \,\,|\,\, \psi=-\dfrac{9k}{4cg}\left(\phi-\dfrac{2}{3}g\right)^{2}+O(\phi^{3}) \right\}
\label{eq:IKYsh-b03}
\end{equation}
and the dynamics of \eqref{eq:IKYsh-f2} near $E_{0}$ is topologically conjugate to the dynamics of the following equation:
\begin{equation}
\phi'=-\dfrac{kg}{c}\left(\phi-\dfrac{2}{3}g\right)^{2}+O(\phi^{3}).
\label{eq:IKYsh-b04}
\end{equation}

\subsection{Existence of the finite equilibria for $b>0$}
\label{sub:IKYsh-f1}
The following result shows the existence of finite equilibria for $b>0$.

\begin{prop}
\label{prop:IKYsh-f1}
Let $c, m, k, g, \varepsilon_{0}, S, V$ be positive constants, and let $b>0$.
Then, the following holds for the function $F(\phi)$ defined in \eqref{eq:IKYsh-f3}:
\begin{enumerate}
\item[(i)] 
$F(0)=F(g)=2^{-1} \varepsilon_{0} SV^{2}$ holds.
\item[(ii)] Equation
\[
5b\phi^{3}-12bg\phi^{2}+3(k+3bg^{2})\phi-2(kg+bg^{3})=0
\]
has a unique real root for $0 < \phi < g$, and we set $\phi^{*}$.
Then, the following relation holds:
\begin{equation}
\dfrac{10}{27}g < \phi^{*}<\dfrac{2}{3}g.
\label{eq:IKYsh-f4}
\end{equation}
\item[(iii)]
The roots to $F(\phi)=0$ for $\phi^{*}\in \mathbb{R}$ in (ii) can be classified as follows:
\begin{itemize}
\item 
If $F(\phi^{*})>0$, $F(\phi)>0$ holds for all $0 < \phi < g$.
In other words, there are no finite equilibria of \eqref{eq:IKYsh-f2}.
\item 
If $F(\phi^{*}) = 0$, then there is exactly one finite equilibrium of \eqref{eq:IKYsh-f2}, which is denoted by $(\phi, \psi)=(\phi^{*}, 0)$.
\item 
If $F(\phi^{*}) < 0$, then there are exactly two solutions that satisfy $F(\phi) = 0$ for $0 < \phi < g$, which is denoted by $\phi_{1}$ and $\phi_{2}$ with $0<\phi_{1}<\phi^{*}<\phi_{2}<g$.
That is, there are exactly two finite equilibria $(\phi, \psi) = (\phi_{1}, 0)$ and $(\phi_{2}, 0)$ of \eqref{eq:IKYsh-f2}.
\end{itemize}
\end{enumerate}
\end{prop}

Proposition \ref{prop:IKYsh-f1} (i) is shown immediately.
We will show (ii) and (iii).
Since 
\begin{equation}
F_{\phi}:=\dfrac{dF}{d\phi}=\phi\cdot \{ 5b\phi^{3}-12bg\phi^{2}+3(k+3bg^{2})\phi-2(kg+bg^{3})\}
\label{eq:IKYsh-f5}
\end{equation}
holds, the point $\phi$ for which $F_{\phi} = 0$ must satisfy either $\phi = 0$ or $G(\phi) = 0$.
Here $G(\phi)$ is defined as follows:
\begin{equation}
G(\phi) := 5b\phi^{3}-12bg\phi^{2}+3(k+3bg^{2})\phi-2(kg+bg^{3}).
\label{eq:IKYsh-f6}
\end{equation}
To prove Proposition \ref{prop:IKYsh-f1} (ii) and (iii), we need to clarify information about the function $G(\phi)$ and its zeros.
The following result can be obtained:

\begin{prop}
\label{prop:IKYsh-f2}
Let $c, m, k, g, \varepsilon_{0}, S, V$ be positive constants, and let $b>0$.
Then, the following holds for the function $G(\phi)$ defined by \eqref{eq:IKYsh-f6}:
\begin{enumerate}
\item[(i)] $G(0)<0$, $G(4g/5)>0$ and $G(g)>0$ hold.
\item[(ii)] For $0 < \phi < g$, there exists exactly one real root $\phi = \phi^{*}\in (0,g)$ such that $G(\phi) = 0$.
\item[(iii)] For any positive constants $b, k,$ and $g$, the following estimate always holds for $\phi^{*}$:
\begin{equation}
0<\dfrac{1}{3}g<\dfrac{10}{27}g<\phi^{*}<\dfrac{2}{3}g<g.
\label{eq:IKYsh-f7}
\end{equation}
\end{enumerate}
\end{prop}
\begin{proof}
From the definition of $G(\phi)$, we obtain (i) since 
\[
G(0)=-2(kg+bg^{3})<0, \quad
G\left(\dfrac{4}{5}g\right)=\dfrac{2}{25}bg^{3}+\dfrac{2}{5}kg>0, \quad
G(g)=kg>0.
\]
For the function $G(\phi)$ defined in \eqref{eq:IKYsh-f6},
\begin{equation}
G_{\phi}:=\dfrac{dG}{d\phi}=15b\phi^{2}-24bg\phi+3(k+3bg^{2})
\label{eq:IKYsh-f8}
\end{equation}
holds.
For values of $\phi$ such that $G_{\phi}=0$, there are three possible outcomes depending on the sign of $bg^{2}-5k$:
\begin{itemize}
\item 
When $bg^{2}-5k<0$, $G_{\phi}>0$ holds for all values of $0<\phi<g$.
Therefore, combined with the result in (i), there exists exactly one real root $\phi=\phi^{*} \in (0,g)$ such that $G(\phi)=0$.
\item 
When $bg^{2}-5k=0$, there is exactly one value of $\phi$ for which $G_{\phi}=0$, and that is $\phi=4g/5$.
Therefore, since $G(4g/5) = 4kg/5 > 0$, and by the result in (i), there exists a unique root $\phi = \phi^{*}$ such that $G(\phi) = 0$.
\item  
When $bg^{2}-5k>0$, there are two values of $\phi\in(0,g)$ such that $G_{\phi}=0$. 
Let these be denoted by $\phi_{\pm}$.
Then we set 
\begin{equation}
\phi_{-}:=\dfrac{4bg- \sqrt{b^{2}g^{2}-5bk}}{5b}, \quad
\phi_{+}:=\dfrac{4bg+ \sqrt{b^{2}g^{2}-5bk}}{5b}.
\label{eq:IKYsh-f9}
\end{equation}
Here, since (i) and $\phi_{-}<4g/5<\phi_{+}$ hold, we see that $G(\phi_{-})>0$.
Furthermore,
\begin{equation}
G(\phi_{+})=-\dfrac{2}{25b}\{-bg(bg^{2}+5k)+\sqrt{b^{2}g^{2}-5bk}(bg^{2}-5k)\}>0
\label{eq:IKYsh-f10}
\end{equation}
holds.
This follows from the fact that $bg(bg^{2}+5k)>0$, $\sqrt{b^{2}g^{2}-5bk}(bg^{2}-5k)>0$ and 
\[
\{bg(bg^{2}+5k)\}^{2}-\{\sqrt{b^{2}g^{2}-5bk}(bg^{2}-5k)\}^{2}>0
\]
holds.
Hence, from $G(\phi_-) > 0$, $G(\phi_+) > 0$, and the results of (i) and (ii), we see that there exists a unique root $\phi = \phi^*$ such that $G(\phi) = 0$ for $0 < \phi < g$.
\end{itemize}
Regarding (iii), since
\[
G\left(\dfrac{10g}{27}\right)=-\dfrac{1156}{19683}bg^{3}-\dfrac{8}{9}kg<0, \quad
G\left(\dfrac{11g}{27}\right)=\dfrac{256}{19683}bg^{3}-\dfrac{7}{9}kg,
\]
it follows that $10g/27 < \phi^*$ holds whenever $b, g, k$  are positive.
Similarly, since $G(2g/3) = 4bg^3/27 > 0$, we obtain (iii).
Thus, this completes the proof of Proposition \ref{prop:IKYsh-f2}.
\end{proof}

We return to the proof of Proposition \ref{prop:IKYsh-f1}.
By Proposition \ref{prop:IKYsh-f2} (ii), there is always exactly one solution to $G(\phi)=0$ for $0<\phi< g$.
Since its value is given by \eqref{eq:IKYsh-f4}, we obtain Proposition \ref{prop:IKYsh-f1} (ii, iii).
Thus, this completes the proof of Proposition \ref{prop:IKYsh-f1}.
\qed
\\

For $b > 0$, we set the parameters and find the unique real root to $G(\phi) = 0$.
According to Proposition \ref{prop:IKYsh-f1}, this root is denoted by $\phi^{*}$.
Then, by substituting $\phi=\phi^{*}$ into $F(\phi)$ and determining its sign, we can determine whether there are 0, 1, or 2 finite equilibria.

\subsection{Existence of the finite equilibria for $b<0$}
\label{sub:IKYsh-bf1}
When $b < 0$, we set $b = -\beta$ and consider the case for $\beta > 0$.
That is, using this value of $\beta > 0$ in $F(\phi)$, as defined in \eqref{eq:IKYsh-f3}, we obtain
\begin{equation}
F(\phi)=-\beta \phi^{5}+3\beta g\phi^{4}+(k-3\beta g^{2})\phi^{3}-(kg-\beta g^{3})\phi^{2}+\dfrac{1}{2}\varepsilon_{0}SV^{2}.
\label{eq:IKYsh-bf1}
\end{equation}
The results regarding the existence of finite equilibria are as follow:

\begin{prop}
\label{prop:IKYsh-bf1}
Let $c, m, k, g, \varepsilon_{0}, S, V$ be positive constants.
Assume that $b=-\beta$ and $\beta>0$.
Then, the following holds for the function $F(\phi)$ defined in \eqref{eq:IKYsh-f3}:
\begin{enumerate}
\item[(i)] $F(0)=F(g)=2^{-1}\varepsilon_{0}SV^{2}$ holds.
\item[(ii)] Equation
\[
5\beta\phi^{3}-12\beta g\phi^{2}-3(k-3\beta g^{2})\phi+2(kg-\beta g^{3})=0
\]
has a unique real root for $0<\phi<g$, and we set $\phi=\phi_{*}$.
Then, the following relation holds:
\begin{equation}
\dfrac{2}{3}g<\phi_{*}<g.
\label{eq:IKYsh-bf2}
\end{equation}
In particular, if $k=\beta g^2$, then $\phi_{*}=\dfrac{6-\sqrt{6}}{5}g$ holds.
\item[(iii)]
The unique real root to $F(\phi)=0$ for $\phi_{*} \in \mathbb{R}$ in (ii) can be classified as follows:
\begin{itemize}
\item 
If $F(\phi_{*})>0$, then $F(\phi)>0$ holds for all $0< \phi <g$.
In other words, there are no finite equilibria of \eqref{eq:IKYsh-f2}.
\item 
If $F(\phi_{*})=0$, then there is exactly one finite equilibrium of \eqref{eq:IKYsh-f2}, which is denoted by $(\phi,\psi)=(\phi_{*}, 0)$.
\item 
If $F(\phi_{*})<0$, then there are exactly two solutions $\tilde{\phi}_{1} \in (0,g)$ and $\tilde{\phi}_{2}\in (0,g)$ that satisfy $F(\phi)=0$.
In addition, $0<\tilde{\phi}_{1}<\phi_{*}<\tilde{\phi}_{2}<g$ holds.
That is, there are exactly two finite equilibria $(\phi, \psi)=(\tilde{\phi}_{1}, 0)$ and $(\tilde{\phi}_{2}, 0)$ of \eqref{eq:IKYsh-f2}.
\end{itemize}
\end{enumerate}
\end{prop}

Proposition \ref{prop:IKYsh-bf1} (i) is shown immediately.
We will show (ii) and (iii).
Since
\begin{equation}
F_{\phi}(\phi):=\dfrac{dF}{d\phi}=-\phi\cdot \{ 5\beta\phi^{3}-12\beta g\phi^{2}-3(k-3\beta g^{2})\phi+2(kg-\beta g^{3}) \}
\label{eq:IKYsh-bf3}
\end{equation}
holds, the point $\phi$ for $F_{\phi}=0$ must satisfy either $\phi=0$ or $\tilde{G}(\phi)=0$.
Here $\tilde{G}(\phi)$ is defined as follows:
\begin{equation}
\tilde{G}(\phi) := 5\beta\phi^{3}-12\beta g\phi^{2}-3(k-3\beta g^{2})\phi+2(kg-\beta g^{3})
\label{eq:IKYsh-bf4}
\end{equation}
To prove Proposition \ref{prop:IKYsh-bf1} (ii) and (iii), we need to clarify information about the function  $\tilde{G}(\phi)$ and its zero points. 
The following result can be obtained.

\begin{prop}
\label{prop:IKYsh-bf2}
Let $c, m, k, g, \varepsilon_{0}, S, V$ be positive constants.
Assume that $\beta>0$ as $b=-\beta$.
Then, the following holds for the function $\tilde{G}(\phi)$ defined in \eqref{eq:IKYsh-bf4}:
\begin{enumerate}
\item[(i)] $\tilde{G}(g)=-kg<0$ holds.
\item[(ii)] There exists exactly one real root $\phi=\phi_{*} \in (0,g)$ such that $\tilde{G}(\phi)=0$.
\item[(iii)] For any positive constants $\beta$, $k$ and $g$, the following estimate always holds for $\phi_{*}$:
\begin{equation}
\dfrac{2}{3}g<\phi_{*}<g.
\label{eq:IKYsh-bf5}
\end{equation}
In particular, if $k=\beta g^2$, then $\phi_{*}=\dfrac{6-\sqrt{6}}{5}g$ holds.
\end{enumerate}
\end{prop}
\begin{proof}
From the definition of $\tilde{G}(\phi)$, we obtain (i).
In addition, we have 
\begin{equation}
\tilde{G}_{\phi}(\phi):=\dfrac{d\tilde{G}}{d\phi}=15\beta\phi^{2}-24\beta g\phi-3(k-3\beta g^{2}).
\label{eq:IKYsh-bf6}
\end{equation}
Based on the relationship between the signs of $\tilde{G}(0)$ and $\tilde{G}_{\phi}(0)$, we obtain the following five cases about the zero points of the function $\tilde{G}_{\phi}$ and the shape of $\tilde{G}$:
\begin{itemize}
\item
When $0<k<\beta g^{2}$, there are two roots of $\phi$ for which $\tilde{G}_{\phi}=0$ holds and these are 
\begin{equation}
\tilde{\phi}_{-}:=\dfrac{4\beta g-\sqrt{\beta(\beta g^2+5k)}}{5\beta}, \quad
\tilde{\phi}_{+}:=\dfrac{4\beta g+\sqrt{\beta(\beta g^2+5k)}}{5\beta}
\label{eq:IKYsh-bf7}
\end{equation}
with
\begin{equation}
0<\tilde{\phi}_{-}<\dfrac{4}{5}g<g<\tilde{\phi}_{+}.
\label{eq:IKYsh-bf8}
\end{equation}
In addition, $\tilde{G}(0)<0$ and $\tilde{G}_{\phi}(0)>0$ holds.
From $\sqrt{\beta(\beta g^{2}+5k)}-\beta g>0$,
\begin{equation}
\tilde{G}(\tilde{\phi}_{-})
=\dfrac{2}{25\beta}\{ \beta^{2}g^{3}+\beta g^{2}\sqrt{\beta(\beta g^{2}+5k)}+5k(\sqrt{\beta(\beta g^{2}+5k)}-\beta g) \}
>0
\label{eq:IKYsh-bf9}
\end{equation}
holds.
Therefore, combined with the result in (i), there exists exactly one real root 
$\phi=\phi_{*} \in (0,g)$ such that $\tilde{G}(\phi)=0$.

\item
When $\beta g^{2} \le k <3\beta g^{2}$, the roots of $\phi$ for which $\tilde{G}_{\phi} = 0$ satisfy \eqref{eq:IKYsh-bf7} and \eqref{eq:IKYsh-bf8}.
Therefore, from \eqref{eq:IKYsh-bf9} and (i), there exists a unique real root $\phi=\phi_{*}$ in $(0,g)$ such that $\tilde{G}(\phi)=0$.

\item
When $k=3\beta g^2$, the values of $\phi$ for which $\tilde{G}_{\phi}=0$ are $\phi=0$ and $\phi=8g/5$.
Therefore, $\tilde{G}(0)>0$ and $\tilde{G}_{\phi}(\phi)<0$ holds in $\phi\in(0,g)$.
From the result of (i), there exists a unique real root $\phi=\phi_{*}$ in $(0,g)$ such that $\tilde{G}(\phi)=0$. 

\item
When $k > 3\beta g^2$, the $\phi$ for which $\tilde{G}_{\phi} = 0$ is $\tilde{\phi}_{+}$ as defined in \eqref{eq:IKYsh-bf7}.
In addition, $\tilde{\phi}_{+} > g$ holds.
By a similar above discussion, the uniqueness of $\phi = \phi_*$ is shown.
\end{itemize}

Thus, this completes the proof of (ii).
For (iii), since $\sqrt{\beta(\beta g^{2}+5k)}-\beta g>0$ holds, we have $\tilde{\phi}_{-}<3g/5$.
Hence, we obtain $\tilde{\phi}_{-}<2g/3$. 
Furthermore, we obtain \eqref{eq:IKYsh-bf5} from $\tilde{G}(2g/3) = 4\beta g^{3}/27 > 0$.
In particular, when $k = \beta g^2$, solving $ \tilde{G}(\phi) = 0$ yields $\phi_{*} = 5^{-1}(6 - \sqrt{6})g$.
Thus, this completes the proof of Proposition \ref{prop:IKYsh-bf2}.
\end{proof}

We return to the proof of Proposition \ref{prop:IKYsh-bf1}.
By Proposition \ref{prop:IKYsh-bf2} (ii), there is always exactly
one solution to $\tilde{G}(\phi)=0$ for $2g/3<\phi< g$.
Hence, we obtain the results of Proposition \ref{prop:IKYsh-bf1}(ii, iii).
Thus, this completes the proof of Proposition \ref{prop:IKYsh-bf1}.
\qed
\\

By Proposition \ref{prop:IKYsh-bf1}, even if $b<0$, we set parameters and find the unique real root to $\tilde{G}(\phi)=0$. 
This root is denoted by $\phi_*$.
Then, by substituting $\phi=\phi_{*}$ into $F(\phi)$ and determining its sign, we can determine whether there are 0, 1, or 2 finite equilibria.

\subsection{Stability of the finite positive equilibria}
\label{sub:IKYsh-sp1}
According to Propositions \ref{prop:IKYsh-f1} and \ref{prop:IKYsh-bf1}, the number of finite equilibria is either 0, 1, or 2, regardless of whether $b > 0$ or $b < 0$.
The purpose of this section is to discuss the local stability of these equilibria. 

\begin{prop}
\label{prop:IKYsh-sp1}
Let $c, m, k, g, \varepsilon_{0}, S, V$ be positive constants and $b\neq 0$.
Then the following holds:
\begin{enumerate}
\item[(i)] 
If $b>0$, the equilibrium $(\phi, \psi)=(\phi^{*}, 0)$ is not hyperbolic.
In addition, the equilibrium $(\phi_{1}, 0)$ is a saddle and $(\phi_{2}, 0)$ is local asymptotically stable.
\item[(ii)] 
If $b<0$, the equilibrium $(\phi, \psi)=(\phi_{*}, 0)$ is not hyperbolic.
In addition, the equilibrium $(\tilde{\phi}_{1}, 0)$ is a saddle and $(\tilde{\phi}_{2}, 0)$ is local asymptotically stable.
\end{enumerate}
\end{prop}
\begin{proof}
We will consider the case where $b > 0$.
For the case where $b < 0$, simply replace $\phi^{*}$ with $\phi_{*}$.
These two cases are essentially the same.

The linearized matrix of the vector field \eqref{eq:IKYsh-f2} at the equilibrium $(\phi, \psi)=(\phi^{*}, 0)$ is 
\[
\left(\begin{array}{cc}
0 & (\phi^{*})^{2} \\ 0 & -cm^{-1}(\phi^{*})^{2}
\end{array}\right).
\]
Since the eigenvalues of this matrix are $0$ and $ -cm^{-1}(\phi^{*})^{2}$, the equilibrium $(\phi, \psi)=(\phi^{*}, 0)$ is not hyperbolic.
The linearized matrices $J_{i}$ ($i=1,2$) of the vector field \eqref{eq:IKYsh-f2} at $(\phi_{1}, 0)$ and $(\phi_{2}, 0)$ are
\[
J_{1}:=\left(\begin{array}{cc}
0 & (\phi_{1})^{2} \\ -m^{-1}\phi_{1}G(\phi_{1}) & -cm^{-1}(\phi_{1})^{2}
\end{array}\right), \quad
J_{2}:=\left(\begin{array}{cc}
0 & (\phi_{2})^{2} \\ -m^{-1}\phi_{2}G(\phi_{2}) & -cm^{-1}(\phi_{2})^{2}
\end{array}\right).
\]
Note that $\phi_{1}<\phi^{*}<\phi_{2}$ holds, we see that $G(\phi_{1})<0$ and $G(\phi_{2})>0$.
Therefore, since the matrix $J_{1}$ has one positive real eigenvalue and one negative real eigenvalue, we conclude that the equilibrium  $(\phi_{1}, 0)$ is a saddle.
Furthermore, since all the real parts of the eigenvalues of the matrix $J_{2}$ are negative, we conclude that the equilibrium $(\phi_{2}, 0)$ is locally asymptotically stable.
\end{proof}

The equilibria $(\phi, \psi)=(\phi^{*}, 0)$ and $(\phi_{*}, 0)$ are not hyperbolic.
The dynamical systems near these equilibria can be understood through the transformation to normal forms and the center manifold theorem (\cite{carr}).

First, we introduce the following transformations for \eqref{eq:IKYsh-f2}:
\[
\tilde{\phi}=\phi-\varphi^{*}, \quad
\tilde{\psi}=\psi
\]
with 
\[
\varphi^{*}:=\begin{cases}
\phi^{*} \quad & {\rm{if}} \quad b>0, \\
\phi_{*} \quad & {\rm{if}} \quad b<0.
\end{cases}
\]
We then have
\begin{equation}
\begin{cases}
\tilde{\phi}'=\tilde{\phi}^{2}\tilde{\psi}+2\varphi^{*}\tilde{\phi}\tilde{\psi}+(\varphi^{*})^{2}\tilde{\psi},
\\
\tilde{\psi}'=
-m^{-1}c(\varphi^{*})^{2}\tilde{\psi}
-2cm^{-1}\varphi^{*}\tilde{\phi}\tilde{\psi}
+f_{1}\tilde{\phi}^{2}
-m^{-1}c\tilde{\phi}^{2}\tilde{\psi}
+f_{2}\tilde{\phi}^{3}
+f_{3}\tilde{\phi}^{4}
-m^{-1}b\tilde{\phi}^{5},
\end{cases}
\label{eq:IKYsh-sp1}
\end{equation}
where $f_{i}$ ($i=1,2,3$) are defined as follows:
\begin{align*}
&f_{1}:=\dfrac{kg+bg^{3}}{m}-3\varphi^{*}\cdot\dfrac{k+3bg^{2}}{m}-\dfrac{10b}{m}(\varphi^{*})^{3}, \\
&f_{2}:=\dfrac{12bg}{m}\varphi^{*}-\dfrac{k+3bg^{2}}{m}-\dfrac{10b}{m}(\varphi^{*})^{2}, \\
&f_{3}:=\dfrac{3bg}{m}-\dfrac{5b}{m}\varphi^{*}.
\end{align*}

Next, focusing on the equilibrium $( \tilde{\phi}, \tilde{\psi}) = (0, 0)$ of the system \eqref{eq:IKYsh-sp1}, the eigenvectors corresponding to the eigenvalues $0$ and $-m^{-1}c(\varphi^{*})^{2}$ are 
\[
{\mathbf{v}}_{1}=(1,0)^{T}, \quad 
{\mathbf{v}}_{2}=(1, -cm^{-1})^{T}.
\]
Let $P=(\mathbf{v}_{1},\mathbf{v}_{2})$.
Then, the equation \eqref{eq:IKYsh-sp1} changes as follows:
\begin{align*}
\left(\begin{array}{cc}
\tilde{\phi}' \\
\tilde{\psi}'
\end{array}
\right) 
&= \left(\begin{array}{cc}
0 & (\varphi^{*})^{2} \\
0 & -cm^{-1}(\varphi^{*})^{2}
\end{array}
\right)\left(\begin{array}{cc}
\tilde{\phi} \\
\tilde{\psi}
\end{array}
\right)+{\mathbf{b}}
\\
&= PP^{-1}\left(\begin{array}{cc}
0 & (\varphi^{*})^{2} \\
0 & -cm^{-1}(\varphi^{*})^{2}
\end{array}
\right)PP^{-1}\left(\begin{array}{cc}
\tilde{\phi} \\
\tilde{\psi}
\end{array}
\right)+{\mathbf{b}}
\\
&= P\left(\begin{array}{cc}
0 & 0 \\
0 & -cm^{-1}(\varphi^{*})^{2}
\end{array}
\right)P^{-1}\left(\begin{array}{cc}
\tilde{\phi} \\
\tilde{\psi}
\end{array}
\right)+{\mathbf{b}},
\end{align*}
where ${\mathbf{b}}$ is defined as follows:
\[
{\mathbf{b}}=\left(\begin{array}{cc}
\tilde{\phi}^{2}\tilde{\psi}+2\varphi^{*}\tilde{\phi}\tilde{\psi}+(\varphi^{*})^{2}\tilde{\psi}
 \\
-2cm^{-1}\varphi^{*}\tilde{\phi}\tilde{\psi}
+f_{1}\tilde{\phi}^{2}
-m^{-1}c\tilde{\phi}^{2}\tilde{\psi}
+f_{2}\tilde{\phi}^{3}
+f_{3}\tilde{\phi}^{4}
-m^{-1}b\tilde{\phi}^{5}
\end{array}
\right)
\]

Next, we set 
\[
\left(\begin{array}{cc}
\hat{\phi} \\
\hat{\psi}
\end{array}
\right)=P^{-1}\left(\begin{array}{cc}
\tilde{\phi} \\
\tilde{\psi}
\end{array}
\right).
\]
We then have
\begin{equation}
\begin{cases}
\hat{\phi}'=m^{-1}c(\varphi^{*})^{2}\hat{\psi}-c^{-1}H(\varphi^{*})\tilde{\phi}^{2}+ {\rm h.o.t}, 
\\
\hat{\psi}'=-2m^{-1}c(\varphi^{*})^{2}\hat{\psi}+c^{-1}H(\varphi^{*})\tilde{\phi}^{2}+ {\rm h.o.t}
\end{cases}
\label{eq:IKYsh-sp2}
\end{equation}
with $\varphi^{*}=\phi^{*}$ or $\varphi^{*}=\phi_{*}$.
Here, the function $H(\varphi^{*})$ is defined as follows:
\begin{equation}
H(\varphi^{*}):= 10b(\varphi^{*})^{3}-18bg(\varphi^{*})^{2}+(9bg^{2}+3k)\varphi^{*}-bg^{3}-kg.
\label{eq:IKYsh-sp3}
\end{equation}
We apply the center manifold theorem for \eqref{eq:IKYsh-sp2}.
It implies that there exists a function $h(\hat{\phi})$ satisfying
\[
h(0)=\dfrac{dh}{d\hat{\phi}}(0)=0
\]
such that the center manifold of the origin for \eqref{eq:IKYsh-sp2} is locally represented as $\{(\hat{\phi},\hat{\psi}) \,|\, \hat{\psi}(s)=h(\hat{\phi}(s))\}$.
Differentiating it with respect to $s$, we obtain that the approximation of the (graph of) center manifold is
\[
\left\{ (\hat{\phi}, \hat{\psi}) \,|\, \hat{\psi}=\dfrac{m}{2c^{2}(\varphi^{*})^{2}}H(\varphi^{*})\cdot \hat{\phi}^{2}+ {\rm h.o.t} \right\}.
\]
Therefore, the dynamics of \eqref{eq:IKYsh-sp2} near the origin is topologically conjugate to the dynamics of the following equation:
\[
\hat{\phi}'=-\dfrac{1}{c}H(\varphi^{*})\cdot \hat{\phi}^{2}+ {\rm h.o.t}.
\]
According to the above discussion, using
\[
\tilde{\phi}=\hat{\phi}+\hat{\psi}, \quad
\tilde{\psi}=-\dfrac{c}{m}\hat{\psi}, \quad
\phi=\tilde{\phi}+\varphi^{*}, \quad
\psi=\tilde{\psi},
\]
the approximation of the (graph of) center manifold is as follows:
\begin{equation}
\left\{ (\phi, \psi) \,\,|\,\, \psi= -\dfrac{H(\varphi^{*})}{2c(\varphi^{*})^{2}}\cdot (\phi-\varphi^{*})^{2}+ {\rm h.o.t} \right\}.
\label{eq:IKYsh-sp4}
\end{equation}
In addition, the dynamics of \eqref{eq:IKYsh-sp2} near the finite equilibrium $(\phi, \psi)=(\varphi^{*}, 0)$ is topologically conjugate to the dynamics of the following equation:
\begin{equation}
\phi'=-\dfrac{H(\varphi^{*})}{2c}(\phi-\varphi^{*})^{2}+ {\rm h.o.t}.
\label{eq:IKYsh-sp5}
\end{equation}

Finally, we can obtain the following result regarding the sign of $H(\varphi^{*})$.

\begin{prop}
\label{prop:IKYsh-sp2}
Let $c, m, k, g, \varepsilon_{0}, S, V$ be positive constants and $b\neq 0$.
Then, the following holds for $H(\varphi^{*})$, as defined by \eqref{eq:IKYsh-sp3}:
\begin{itemize}
\item If $b>0$, $H(\phi^{*})>0$ holds for $10g/27<\phi^{*}<2g/3$.
\item If $b<0$, $H(\phi_{*})>0$ holds for $2g/3<\phi_{*}<g$.
\end{itemize}
\end{prop}
\begin{proof}
When $b > 0$, then $\varphi^{*} = \phi^{*}$, and when $b < 0$, $\varphi^{*} = \phi_{*}$.
When $b > 0$, then $\phi^{*}$ is a solution to $G(\phi) = 0$ as defined in \eqref{eq:IKYsh-f6}, and when $b < 0$, then $\phi_{*}$ is a solution to $\tilde{G}(\phi) = 0$ as defined in \eqref{eq:IKYsh-bf4}.
Furthermore, $G(\phi^{*}) = \tilde{G}(\phi_{*}) = 0$ holds.
Hereinafter, we will consider the behavior of the function $H(\varphi^{*})$ under the condition $F(\varphi^{*}) = G(\varphi^{*}) = 0$.
From $F(\varphi^{*})=G(\varphi^{*})=0$, we have
\begin{equation}
H(\varphi^{*}) =3b(\varphi^{*})^{3}-3bg(\varphi^{*})^{2}+\dfrac{3}{2}\varepsilon_{0}SV^{2} (\varphi^{*})^{-2}
= \dfrac{3b(\varphi^{*})^{5}-3bg(\varphi^{*})^{4}+\dfrac{3}{2}\varepsilon_{0}SV^{2}}{(\varphi^{*})^{2}}.
\label{eq:IKYsh-sp6}
\end{equation}
Therefore, the problem of determining the sign of the function $H(\varphi^{*})$ reduces to the problem of determining the sign of the function $\hat{H}(\varphi^{*})$, which is defined as follows:
\begin{equation}
\hat{H}(\varphi^{*}):=3b(\varphi^{*})^{5}-3bg(\varphi^{*})^{4}+\dfrac{3}{2}\varepsilon_{0}SV^{2}.
\label{eq:IKYsh-sp7}
\end{equation}
Here, $\hat{H}(0)=\hat{H}(g)=2^{-1}3\varepsilon_{0}SV^{2}>0$ and 
\[
\hat{H}_{\varphi^{*}}(\varphi^{*}):= \dfrac{d\hat{H}}{d\varphi^{*}}=3b(\varphi^{*})^{3}(5\varphi^{*}-4g)
\]
and $\hat{H}_{\varphi^{*}}(0)=0$, $\hat{H}_{\varphi^{*}}(g)=3bg^{4}$ holds.
The equation $\hat{H}_{\varphi^{*}}(\varphi^{*})=0$ holds if and only if $\varphi^{*}=0$ or $\varphi^{*}=4g/5$.
In addition,
\begin{equation}
\hat{H}(4g/5):=-\dfrac{768}{3125}bg^{5}+\dfrac{3}{2}\varepsilon_{0}SV^{2}
\label{eq:IKYsh-sp8}
\end{equation}
holds.
\begin{enumerate}
\item[(i)]
When $b>0$, then the graph of $\hat{H}(\varphi^{*})$ varies depending on the sign of $\hat{H}(4g/5)$.
\begin{itemize}
\item If $\hat{H}(4g/5)>0$, then $\hat{H}(\phi^{*})>0$ holds for $10g/27<\phi^{*}<2g/3$.
\item If $\hat{H}(4g/5)=0$, then $\hat{H}(\phi^{*})>0$ holds from $10g/27<\phi^{*}<2g/3<4g/5$.
\item  We consider the case in $\hat{H}(4g/5)<0$.
It is sufficient to show that $\hat{H}(2g/3)>0$.
We will therefore prove this by contradiction.
That is, we will assume that $\hat{H}(2g/3)<0$.
Then 
\begin{equation}
\dfrac{1}{2}\varepsilon_{0}SV^{2}< \dfrac{16}{243}bg^{5}
\label{eq:IKYsh-sp9}
\end{equation}
holds.
Since $\varphi^{*}=\phi^{*}$ satisfies $F(\phi)=0$, applying \eqref{eq:IKYsh-sp9} yields 
\begin{equation}
(\phi^{*})^{2}\hat{F}(\phi^{*})=
(\phi^{*})^{2}\left\{ \dfrac{259}{243}b(\phi^{*})^{3}-3bg(\phi^{*})^{2}+(k+3bg^{2})\phi^{*}-(kg+bg^{3}) \right\}>0
\label{eq:IKYsh-sp10}
\end{equation}
with 
\[
\hat{F}(\phi^{*})=\dfrac{259}{243}b(\phi^{*})^{3}-3bg(\phi^{*})^{2}+(k+3bg^{2})\phi^{*}-(kg+bg^{3}) .
\]
Here, $\hat{F}(0)<0$ and $\hat{F}(g)>0$ and 
\[
\hat{F}(2g/3)=-\dfrac{115}{6561}bg^{3}-\dfrac{1}{3}kg<0
\]
holds.
Furthermore, since $d\hat{F}/d\phi^* > 0$ for $0 \leq \phi^* \leq g$, it follows that $\hat{F}(\phi^*) < 0$ for $10g/27 < \phi^* < 2g/3$.
This contradicts \eqref{eq:IKYsh-sp10}.
Thus, we obtain $\hat{H}(2g/3) > 0$.
Therefore, $\hat{H}(\phi^{*}) > 0$ holds for $10g/27 < \phi^{*} < 2g/3$.
\end{itemize}
\item[(ii)]
We can see that $\hat{H}_{\phi_{*}}(\phi_{*}) > 0$ for $b < 0$.
This is shown by the fact that $\hat{H}_{\phi_{*}}(\phi_{*}) = 0$ holds if and only if $\phi_{*} = 0$ or $\phi_{*} = 4g/5$ holds.
Then we obtain $\hat{H}(4g/5) > \hat{H}(0) = \hat{H}(g) > 0$.
\end{enumerate}
Thus, the proof of Proposition \ref{prop:IKYsh-sp2} regarding $H(\varphi^{*})$, which is defined in \eqref{eq:IKYsh-sp3}, is complete.
\end{proof}

\section{Dynamics at infinity for $b=0$}
\label{sec:IKYsh-d}
In this section, we analyze the dynamical systems at infinity for \eqref{eq:IKYsh-int8} and \eqref{eq:IKYsh-f2} with $b=0$. 
To achieve this, we first introduce the definition of Poincar\'e-type compactification and its local coordinate transformations.

\subsection{Short review of Poincar\'e-type compactifications inducing dynamics at infinity}
\label{sub:IKYsh-d1}
Let us examine the dynamics of equation \eqref{eq:IKYsh-f2} at infinity without considering the initial conditions.
The system \eqref{eq:IKYsh-f2} is an asymptotically quasi-homogeneous vector field of type $(1,1)$ and order $3$ at infinity for $b=0$.
See \cite{Matsue1, Matsue2, UBTW} for the definition and details.
This notion guarantees that the structure of a dynamical system at infinity can be correctly extracted.
Matsue (\cite{Matsue1, Matsue2}) has shown that the scale invariance known as the homogeneity of the vector field must hold at infinity.

Full dynamics $\Phi=\mathbb{R}^{2}\, \cup\, \{(\phi,\psi) \mid \|(\phi,\psi)\|=+\infty \}$ of \eqref{eq:IKYsh-f2}, including those at infinity, are induced by Poincar\'e-type compactification.
$\Phi$ is hereinafter referred to as a Poincar\'e-type disk.
The outline of this method is described below.
See \cite{FAL, QTW, UPKPP, UBTW, Matsue1, Matsue2} and the references therein for details.

Consider a plane $(y_{1},y_{2},y_{3})=(\phi,\psi,1)$ and a sphere in $\mathbb{R}^{3}$, along with their partitions ($k=1,2,3$):
\begin{align*}
\mathbb{S}^{2} &= \{ y \in \mathbb{R}^{3} \, |\, y_{1}^{2} + y_{2}^{2}+y_{3}^{2}=1\}, \quad
\mathbb{S}^{1}=\{y \in \mathbb{S}^{2}\, | \, y_{3}=0\},\\
H_{+} &= \{ y \in \mathbb{S}^{2}\,|\,y_{3}>0\}, \quad
H_{-} = \{ y \in \mathbb{S}^{2}\,|\,y_{3}<0\}, \\
U_{k} &= \{y \in \mathbb{S}^{2} \, | \, y_{k}>0\}, \quad
V_{k} = \{y \in \mathbb{S}^{2} \, | \, y_{k}<0\}.
\end{align*}
Let $\Delta(\phi,\psi)$ be defined as $\Delta(\phi,\psi) = \sqrt{\phi^{2}+\psi^{2}+1}$.
Furthermore, we define the map $f^{\pm}$ from $(\phi, \psi)\in \mathbb{R}^{2}$ to $H^{\pm}$ as 
\[
f^{\pm}(\phi,\psi):= \pm \left( \dfrac{\phi}{\Delta(\phi,\psi)},\dfrac{\psi}{\Delta(\phi,\psi)},\dfrac{1}{\Delta(\phi,\psi)} \right)
\]
and the projection $g^{\pm}_{k}$:
\[
g^{+}_{k} : U_{k} \to \mathbb{R}^{2} \quad {\rm and} 
\quad g^{-}_{k} : V_{k} \to \mathbb{R}^{2} 
\]
as 
\[
g^{+}_{k}(y_{1},y_{2},y_{3}) = - g^{-}_{k}(y_{1},y_{2},y_{3})
 = (y_{m}/y_{k}, y_{n}/y_{k})
 \]
 for $m<n$ and $m,n\neq k$.
The projected vector fields are obtained as the vector fields on the planes
 \begin{align*}
\overline{U}_{k} = \{y \in \mathbb{R}^{3} \, | \, y_{k} = 1\}, \\
\overline{V}_{k} = \{y \in \mathbb{R}^{3} \, | \, y_{k} = -1\}
\end{align*}
for each local charts $U_{k}$ and $V_{k}$ ($k=1,2,3$).
In addition, we define $(x, \lambda) := g^{\pm}_{k}(y)$.

For instance, we consider the case $k=1$.
It follows that the projection from $\mathbb{R}^2$ to $\overline{U}_1$ as 
\[
(g^{+}_{1} \circ f^{+})(\phi,\psi) 
= \left ( \dfrac{\psi}{\phi},\dfrac{1}{\phi}\right) = (x, \lambda).
\]
See also Figure \ref{fig:IKYsh-d1}.
Therefore, we can obtain the dynamics on the local chart $\overline{U}_{1}$ by the change of variables $\phi=1/\lambda$ and $\psi=x/\lambda$.
Throughout this paper, we follow the notations used here for the Poincar\'e-type compactification.
Since we will be considering $\phi\ge 0$ in the following sections, it is sufficient to consider its dynamics in the local coordinates $\overline{U}_{j}$ ($j=1,2$) and $\overline{V}_{2}$.

\begin{figure}[t]
\begin{center}
\includegraphics[scale=0.3]{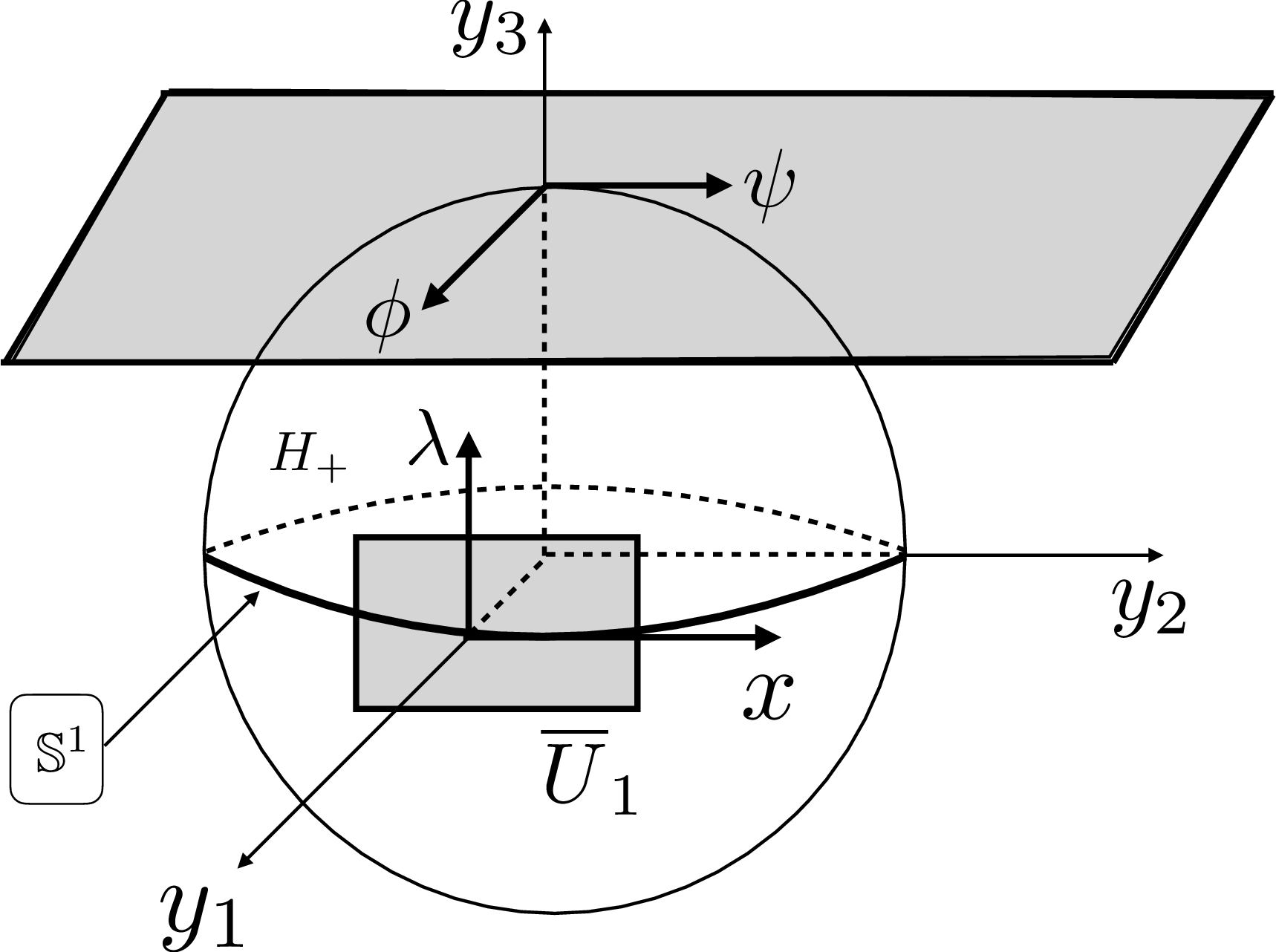}
\caption{Schematic pictures of the locations of the chart $\overline{U}_1$.}
\label{fig:IKYsh-d1}
\end{center}
\end{figure}

\subsection{Dynamics on the local chart $\overline{U}_{1}$}
\label{sub:IKYsh-d2}
To obtain the dynamics on the chart $\overline{U}_{1}$, we introduce the coordinates $(\lambda, x)$ given by
\[
\phi=1/\lambda, \quad
\psi=x/\lambda.
\]
By using the time-scale desingularization $d\tau/ds=\lambda^{-2}$, we obtain
\begin{equation}
\begin{cases}
\lambda_{\tau}=-\lambda x, \\
x_{\tau}= -m^{-1}cx-m^{-1}k+m^{-1}kg\lambda -(2m)^{-1}\varepsilon_{0}SV^{2}\lambda^{3}-x^{2}
\end{cases}
\label{eq:IKYsh-d1}
\end{equation}
with $\lambda_{\tau}=d\lambda/d\tau$ and $x_{\tau}=dx/d\tau$.
The equilibria of \eqref{eq:IKYsh-d1} on $\{\lambda=0\}$ are
\[
(\lambda, x)=(0, x_{-}), \,\, (0, x_{+}),
\]
which correspond to the equilibria at infinity.
Here, $x_{\pm}$ are defined as follows:
\[
x_{\pm}=(2m)^{-1} [c\pm \sqrt{c^{2}-4km}].
\]
Therefore, the existence of an equilibrium depends on the sign of $c^{2}-4km$ and can be classified as follows:
\begin{enumerate}
\item[(i)] If $c^{2}-4km<0$, then there is no equilibrium of \eqref{eq:IKYsh-d1}.
\item[(ii)] If $c^{2}-4km=0$, then \eqref{eq:IKYsh-d1} has exactly one equilibrium $(\lambda, x)=(0, -(2m)^{-1}c)$.
\item[(iii)] If $c^{2}-4km>0$, then \eqref{eq:IKYsh-d1} has exactly two equilibria $(\lambda, x)=(0, x_{-})$ and $(0, x_{+})$ with $x_{-}<x_{+}<0$.
\end{enumerate}
Next, we will present the results concerning the local stability of these equilibria.
The linearized matrices of the vector field \eqref{eq:IKYsh-d1} at each equilibrium for $c^{2}-4km>0$ are
\[
(0, x_{-}): \left(\begin{array}{cc}
-x_{-} & 0 \\ \dfrac{kg}{m} & -\sqrt{\dfrac{c^{2}-4km}{m^{2}}}
\end{array}\right),
\quad
(0, x_{+}): \left(\begin{array}{cc}
-x_{+} & 0 \\ \dfrac{kg}{m} & \sqrt{\dfrac{c^{2}-4km}{m^{2}}}
\end{array}\right).
\]
By calculating the eigenvalues, we can see that the equilibrium $(0,x_{-})$ is a saddle and  $(0, x_{+})$ is unstable.

On the other hand, in the case of $c^{2}-4km=0$, the linearized matrix corresponding to the equilibrium $(\lambda, x)=(0, -(2m)^{-1}c)$ has a zero eigenvalue. 
Hence, we cannot immediately obtain the dynamical system near this equilibrium.
From \eqref{eq:IKYsh-d1}, 
\[
\left. \lambda_{\tau} \vphantom{\big|}\right|_{\lambda>0,\,\, x<0}>0
\]
holds as nullcline.
In other words, even as time approaches infinity, the solution to this system will not be sink into the equilibrium $(\lambda, x)=(0, -(2m)^{-1}c)$.
Therefore, a trajectory starting from the initial conditions $(\phi,\psi)=(g, 0)$ cannot reach this equilibrium at infinity.

\subsection{Dynamics on the local chart $\overline{U}_{2}$}
\label{sub:IKYsh-d3}
The change of coordinates
\[
\phi=x/\lambda, \quad \psi=1/\lambda
\]
and the time-rescaling $d\tau/ds=\lambda^{-2}$ give the projection dynamics of \eqref{eq:IKYsh-f2} on the local chart $\overline{U}_{2}$.
Then we have 
\begin{equation}
\begin{cases}
\lambda_{\tau}= m^{-1}c\lambda x^{2}+m^{-1}k\lambda x^{3}-m^{-1}kg \lambda^{2}x^{2}+(2m)^{-1}\varepsilon_{0}SV^{2}\lambda^{4} 
\\
x_{\tau}=x^{2}+m^{-1}cx^{3}+m^{-1}kx^{4}-m^{-1}kg \lambda x^{3}+(2m)^{-1}\varepsilon_{0}SV^{2}\lambda^{3}x.
\end{cases}
\label{eq:IKYsh-d2}
\end{equation}
The region under consideration is $\{(\lambda, x) \mid \lambda\ge0,\,\,x\ge0\}$.
For $\{\lambda=0\}$, the equilibrium of the system \eqref{eq:IKYsh-d2} that does not coincide with $\overline{U}_{1}$ is 
\[
(\lambda, x)=(0,0).
\]
Since the eigenvalues of the linearized matrix corresponding to this equilibrium are double zero eigenvalues, it is not a hyperbolic equilibrium.
Using the nullcline, the following holds:
\[
\left. \lambda_{\tau} \vphantom{\big|}\right|_{\lambda>0,\,\, x=0}>0, \quad
\left. x_{\tau} \vphantom{\big|}\right|_{\lambda=0,\,\, x>0}>0
\]
It can be seen that a trajectory starting in the first quadrant of the $\lambda x$ plane does not go to the origin or any of the second, third, or fourth quadrants.

\subsection{Dynamics on the local chart $\overline{V}_{2}$}
\label{sub:IKYsh-d4}
The change of coordinates $\phi=-x/\lambda$ and $\psi=-1/\lambda$ and time-rescaling $d\tau/ds=\lambda^{-2}$ to study the dynamics on the local chart $\overline{V}_{2}$ yields
\begin{equation}
\begin{cases}
\lambda_{\tau}= m^{-1}c\lambda x^{2}+m^{-1}k\lambda x^{3}+m^{-1}kg \lambda^{2}x^{2}-(2m)^{-1}\varepsilon_{0}SV^{2}\lambda^{4}, 
\\
x_{\tau}=x^{2}+m^{-1}cx^{3}+m^{-1}kx^{4}+m^{-1}kg \lambda x^{3}-(2m)^{-1}\varepsilon_{0}SV^{2}\lambda^{3}x.
\end{cases}
\label{eq:IKYsh-d3}
\end{equation}
The range corresponding to $\{\phi\ge 0\}$ is $\{(\lambda, x) \mid \lambda\ge 0,\,\,x \le 0\}$.
The equilibrium of this system that does not coincide with $\overline{U}_{1}$ is
\[
(\lambda, x)=(0,0)
\]
and the Jacobian matrices of the vector filed \eqref{eq:IKYsh-d3} at origin is
\[
\left(\begin{array}{cc}
0 & 0 \\ 0 & 0
\end{array}\right).
\]
Hence, this equilibrium is not hyperbolic.
Since the nullcline cannot determine the dynamics near the origin, the following desingularization of vector field by the blow-up technique is an effective method to study the behavior near its equilibrium:
\[
\lambda=r\bar{\lambda},\quad
x=r^{3}\bar{x}.
\]
For instance, see \cite{AFJ, FAL} and the references therein for details about the blow-up technique.
Since we are interested in the dynamics on $\Phi_{\ge 0}=\Phi \,\cup\, \{\phi\ge 0\}$, it is sufficient to consider the dynamics in the blow-up vector field with local coordinates $\bar{\lambda} = 1$ and $\bar{x} = -1$.

\subsubsection{Dynamics on the chart $\{\bar{\lambda}=1\}$}
\label{sub:IKYsh-d4-1}
By the change of coordinates $\lambda=r$, $x=r^{3}\bar{x}$ and the time-rescaling $d\sigma/d\tau=r^{3}$, we have
\begin{equation}
\begin{cases}
r_{\sigma}=m^{-1}cr^{4}\bar{x}^{2}+m^{-1}kr^{7}\bar{x}^{3}+m^{-1}kgr^{5}\bar{x}^{2}-(2m)^{-1}\varepsilon_{0}SV^{2}r, \\
\bar{x}_{\sigma}=\bar{x}^{2}-m^{-1}2cr^{3}\bar{x}^{3}-m^{-1}2kr^{6}\bar{x}^{4}-m^{-1}2kgr^{4}\bar{x}^{3}+m^{-1}\varepsilon_{0}SV^{2}\bar{x},
\end{cases}
\label{eq:IKYsh-d4}
\end{equation}
where $r_{\sigma}=dr/d\sigma$ and $\bar{x}=d\bar{x}/d\sigma$.
The equilibria on $\{r=0, \,\,\bar{x}\le 0\}$ are 
\[
(r, \bar{x})=(0,0), \quad E^{*}: (0, -m^{-1}\varepsilon_{0}SV^{2}).
\]
Calculating the eigenvalues of corresponding Jacobian matrices, we can conclude that the equilibrium $(0,0)$ is a saddle, and $E^{*}: (0, -m^{-1}\varepsilon_{0}SV^{2})$ is a sink.
In particular, the eigenvalues of the Jacobian matrix for $E^{*}$ are the multiple roots $-(2m)^{-1} \varepsilon_{0} SV^{2}$.

\subsubsection{Dynamics on the chart $\{\bar{x}=-1\}$}
\label{sub:IKY-d4-2}
By the change of coordinates $\lambda=r\bar{\lambda}$, $x=-r^{3}$ and the time-rescaling $d\sigma/d\tau=r^{3}$, we obtain the following:
\begin{equation}
\begin{cases}
r_{\sigma}=-3^{-1}r+(3m)^{-1}cr^{4}-(3m)^{-1}kr^{7}+(3m)^{-1}kgr^{5}\bar{\lambda}-(6m)^{-1}\varepsilon_{0}SV^{2}r\bar{\lambda}^{3}, \\
\bar{\lambda}_{\sigma}=3^{-1}\bar{\lambda}+(3m)^{-1}2cr^{3}\bar{\lambda}-(3m)^{-1}2kr^{6}\bar{\lambda}+(3m)^{-1}2kgr^{4}\bar{\lambda}^{2}-(3m)^{-1}\varepsilon_{0}SV^{2}\bar{\lambda}^{4}.
\end{cases}
\label{eq:IKYsh-d5}
\end{equation}
For $\{r=0, \,\,\bar{\lambda}\ge 0\}$, the equilibrium of this system that does not ciincide with $\{\bar{\lambda}=1\}$ is 
\[
(r, \bar{\lambda})=(0,0).
\]
By the further computations, we can see that the equilibrium $(0,0)$ is a saddle.
Combining the dynamics on the charts, we understand the dynamics of the blow-up vector fields and $\overline{V}_{2}$ (see Figure \ref{fig:IKYsh-d2}).

\begin{figure}[t]
\centering
\includegraphics[scale=0.3]{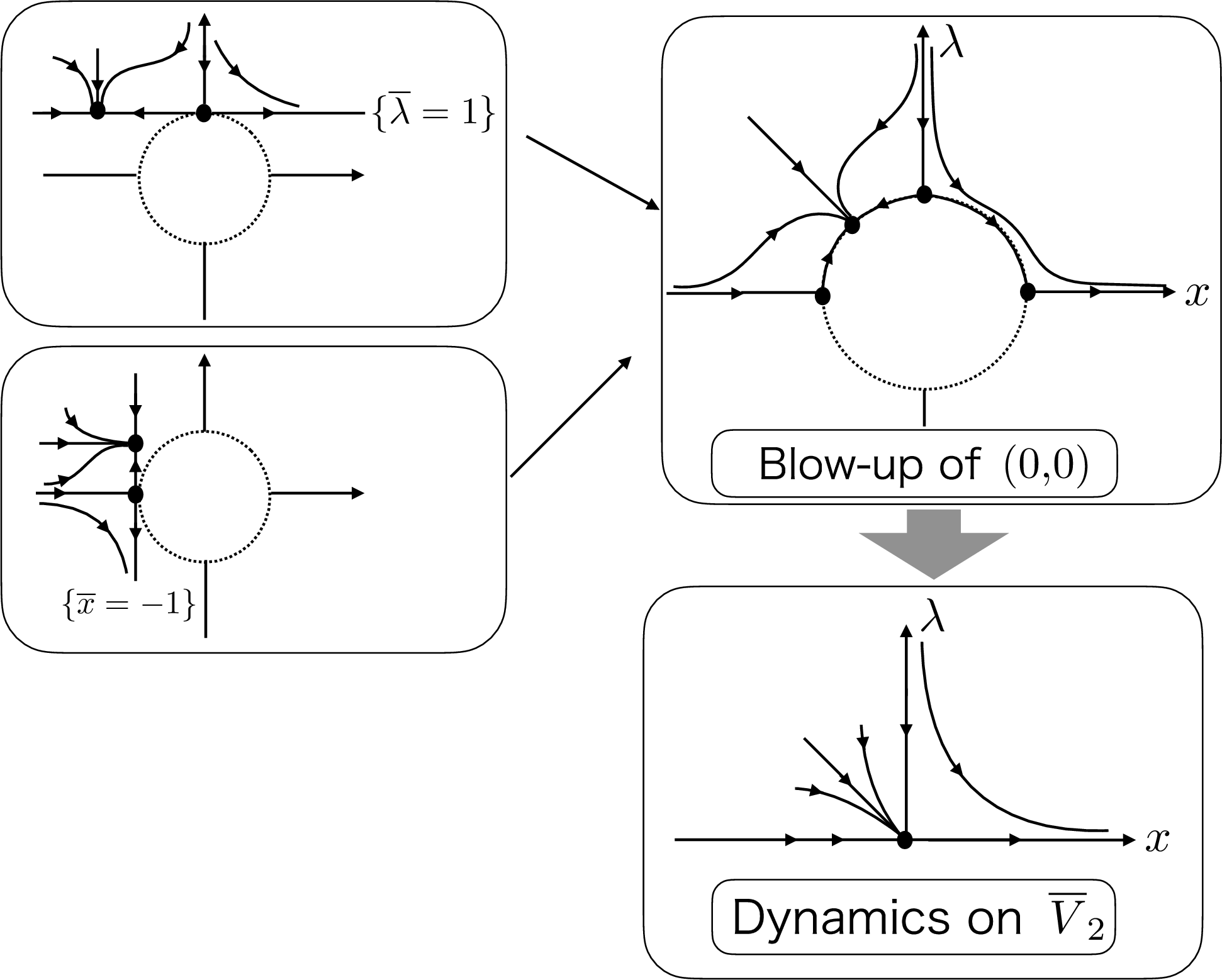}
\caption{Schematic pictures of the dynamics of the blow-up vector fields and $\overline{V}_{2}$.}
\label{fig:IKYsh-d2}
\end{figure}

\subsection{Dynamics and connecting orbits on the Poincar\'e-type disk}
\label{sub:IKYsh-d5}
Combining the dynamics on the charts $\overline{U}_{j}$ ($j=1,2$) and $\overline{V}_{2}$, we obtain the full dynamics $\Phi_{\ge 0}$ that is equivalent to the dynamics of \eqref{eq:IKYsh-f2} or \eqref{eq:IKYsh-int8} for $b=0$. 
We combine the information on the finite equilibria and the equilibria at infinity, we obtain Figure \ref{fig:IKYsh-d3}.
Note that on the circle in Figure \ref{fig:IKYsh-d3}, this corresponds to $\|(\phi, \psi)\| = +\infty$.
Figure \ref{fig:IKYsh-d3} shows the case when $K > 0$.
The only difference for $K = 0$ and $K < 0$ is the point at which the number of finite equilibria changes.

\begin{figure}[t]
\centering
\includegraphics[scale=0.3]{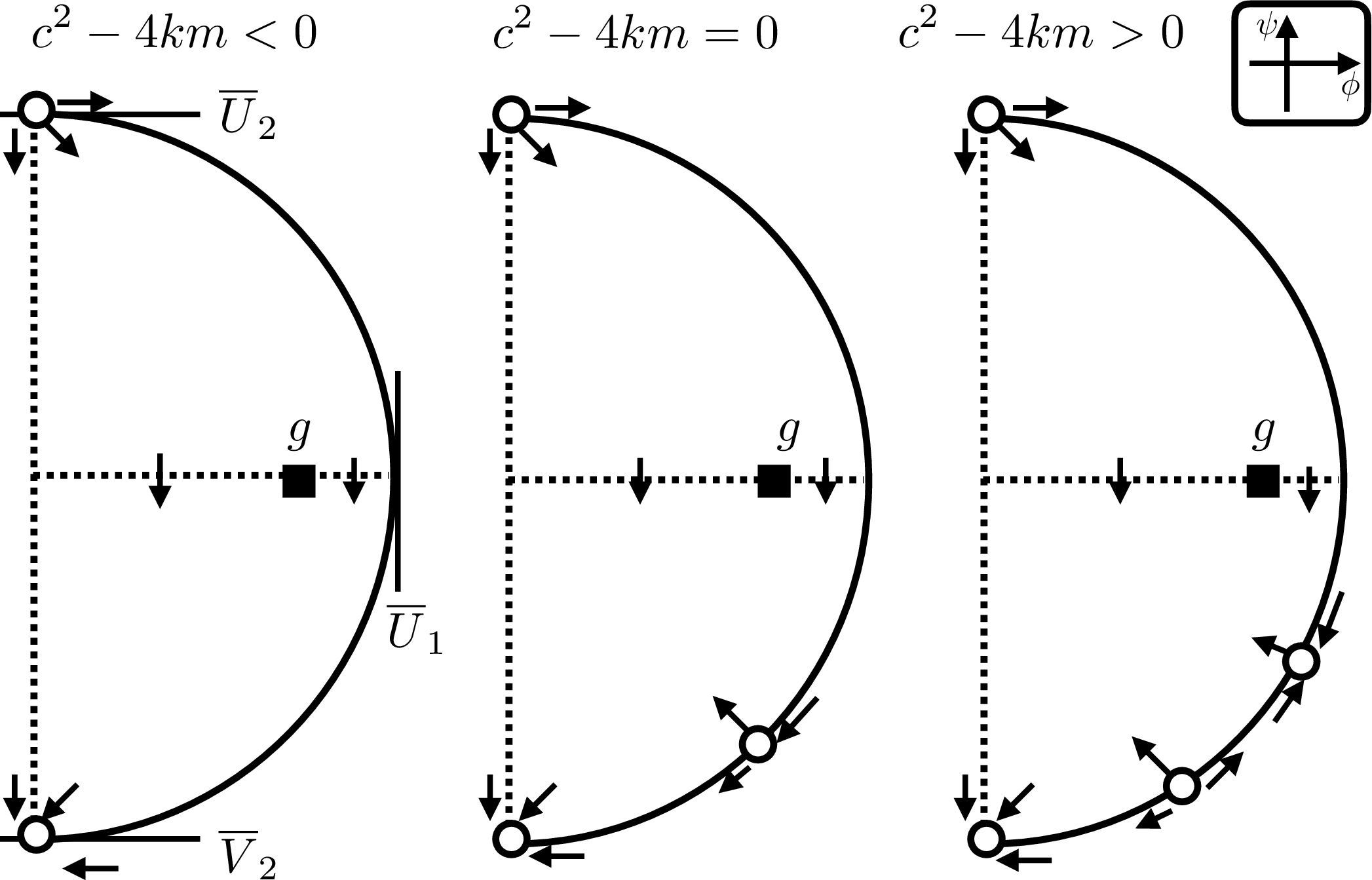}
\caption{Schematic pictures of the dynamics on a Poincar\'e-type disk for $K > 0$. The white circles represent equilibria, and the black squares represent the initial conditions imposed  on \eqref{eq:IKYsh-f2} or \eqref{eq:IKYsh-int8}.}
\label{fig:IKYsh-d3}
\end{figure}

In this subsection, we discuss a orbit starting from the initial conditions $(\phi, \psi)=(g,0)$.
Depending on the value of $K$, the orbit heads toward different points.
When $K < 0$, the objective is to show that the orbit goes to the point $(\phi, \psi) = (\phi_{2}, 0)$, and when $K = 0$, the orbit heads toward $(\phi, \psi) = (2g/3, 0)$ and when $K > 0$, it heads toward $(\phi, \psi) = (0, -\infty)$.

First, when $K > 0$, considering the flow along the $\phi$-axis as a nullcline.
\[
\left. \psi' \vphantom{\big|}\right|_{\psi=0}<0
\]
holds.
Therefore, an orbit starting from the initial conditions $(\phi, \psi)=(g,0)$ must not leave the region $\Phi_{\ge 0} \,\cap\, \{\phi>0,\,\, \psi<0\}$.
According to the Poincar\'e-Bendixson theorem (see, for instance, \cite{Wiggins} and the references therein), it follows that a trajectory starting from the initial conditions $(\phi, \psi)=(g,0)$ must approach the equilibrium at infinity $(\phi, \psi)=(0, -\infty)$.
Using a similar argument, since $g$ is small, there exists a connecting orbit between the point corresponding to the initial conditions and $(\phi, \psi)=(\phi_{2}, 0)$ when $K<0$.
When $K=0$, we see that there exists a connecting orbit between the initial point and $(\phi, \psi)=(2g/3, 0)$.

\section{Dynamics at infinity for $b\neq 0$}
\label{sec:IKYsh-dib}
When $b \neq 0$, \eqref{eq:IKYsh-f2} is an asymptotically quasi-homogeneous vector field of type $(1,2)$ and order $3$ at infinity.
The reason the type and order differ from the case where $b=0$ is that, in \eqref{eq:IKYsh-f2}, the term of highest degree $\phi^{5}$ with $b$ as a coefficient becomes dominant.
Therefore, note that the dynamics at infinity differ between the cases for $b=0$ and $b \neq 0$.

\subsection{Dynamics on the local chart $\overline{U}_{1}$}
\label{sub:IKYsh-dib1}
Reflecting the type $(1,2)$ in the exponent of $\lambda$, the change of coordinates $\phi=1/\lambda$ and $\psi=x/\lambda^{2}$ and time-rescaling $d\tau/ds=\lambda^{-2}$ yield
\begin{equation}
\begin{cases}
\lambda_{\tau}=-\lambda x, \\
x_{\tau}= -m^{-1} [c\lambda x-b+3gb\lambda-(3g^{2}b+k)\lambda^{2}+g(k+bg^{2})\lambda^{3}-2^{-1}\varepsilon_{0}SV^{2}\lambda^{5}]-2x^{2}.
\end{cases}
\label{eq:IKYsh-dib1}
\end{equation}
Although the equilibrium of system \eqref{eq:IKYsh-dib1} on the $\{\lambda=0\}$ does not exist for $b>0$, it has the following for $b<0$:
\[
(\lambda, x)=(0, \pm\sqrt{-(2m)^{-1}b}).
\]
The Jacobian matrix of the vector field at this equilibrium is 
\[
(0, \pm\sqrt{-(2m)^{-1}b}): \left(\begin{array}{cc}
\mp\sqrt{-(2m)^{-1}b} & 0 \\ -m^{-1}c\sqrt{-(2m)^{-1}b}+3m^{-1}gb & \mp\sqrt{-(2m)^{-1}b}
\end{array}\right).
\]
By calculating the eigenvalues, we can see that the equilibrium $(0, -\sqrt{-(2m)^{-1}b})$ is unstable and $(0, \sqrt{-(2m)^{-1}b})$ is an asymptotically stable.

\subsection{Dynamics on the local chart $\overline{V}_{2}$}
\label{sub:IKYsh-dib2}
By the change of coordinates $\phi=-x/\lambda$, $\psi=-1/\lambda^{2}$ and the time-rescaling $d\tau/ds=\lambda^{-3}$, 
\begin{equation}
\begin{cases}
\lambda_{\tau}
=(2m)^{-1} [c\lambda^{2}x^{2}+b\lambda x^{5}+3gb \lambda^{2}x^{4}+(3g^{2}b+k)\lambda^{3}x^{3} \\
\quad\quad\quad +g(k+bg^{2})\lambda^{4}x^{2}-2^{-1}\varepsilon_{0}SV^{2}\lambda^{6}],
\\
x_{\tau}
=x^{2}+(2m)^{-1}[c\lambda x^{3}+bx^{6}+3gb \lambda x^{5}+(3g^{2}b+k)\lambda^{2}x^{4} \\
\quad\quad\quad +g(k+bg^{2})\lambda^{3}x^{3}-2^{-1}\varepsilon_{0}SV^{2}\lambda^{5}x]
\end{cases}
\label{eq:IKYsh-dib2}
\end{equation}
holds.
In addition to the equilibrium $(\lambda, x)=(0,0)$ on $\{\lambda = 0\}$, the following point exists only when $b < 0$:
\[
(\lambda, x)=(0, \pm[ 2m(-b)^{-1}]^{1/4}).
\]
Since our focus is on $\phi \ge 0$, it is sufficient to consider only $( \lambda, x) = (0, -[2m(-b)^{-1}]^{1/4})$.
The linearized matrices at these equilibria are
\begin{align*}
&(0, 0): \left(\begin{array}{cc}
0 & 0 \\ 0 & 0
\end{array}\right), \\
&(0, -[ 2m(-b)^{-1}]^{1/4}): \left(\begin{array}{cc}
-(2m)^{-1}b[2m(-b)^{-1}]^{5/4} & 0 \\ e_1 & 4[2m(-b)^{-1}]^{1/4}
\end{array}\right),
\end{align*}
with 
\[
e_1=-(2m)^{-1}c[2m(-b)^{-1}]^{3/4}-(2m)^{-1}3gb[2m(-b)^{-1}]^{5/4}.
\]
Hence, the equilibrium $(\lambda, x)=(0, -[2m(-b)^{-1}]^{1/4})$ is unstable.
On the other hand, the origin is not hyperbolic. To determine the dynamics near the origin, we introduce the following blow-up coordinates, as in the previous section:
\[
\lambda=r\bar{\lambda}, \quad x=r^{5}\bar{x}
\]

Using the coordinates $\lambda=r$ and $x=r^{5}\bar{x}$ and the time-rescaling $d\sigma/d\tau=r^{5}$, the dynamics on the local chart $\{\bar{\lambda}=1\}$ is as follows:
\begin{equation}
\begin{cases}
r_{\sigma}
= (2m)^{-1} \left[ cr^{7}\bar{x}^{2}+br^{21}\bar{x}^{5}+3gb r^{17}\bar{x}^{4}+(3g^{2}b+k)r^{13}\bar{x}^{3} \right. \\
\quad\quad\quad \left. +(kg+bg^{3})r^{9}\bar{x}^{2}-2^{-1}\varepsilon_{0}SV^{2}r \right],
\\
\bar{x}_{\sigma}
=\bar{x}^{2}-(m)^{-1} \left[2cr^{6}\bar{x}^{3}+2br^{20}\bar{x}^{6}+6gb r^{16}\bar{x}^{5}+2(3g^{2}b+k)r^{12}\bar{x}^{4} \right. \\
\quad\quad\quad  \left.+2(kg+bg^{3})r^{8}\bar{x}^{3}-\varepsilon_{0}SV^{2}\bar{x} \right].
\end{cases}
\label{eq:IKYsh-dib3}
\end{equation}
The equilibria on $\{r=0\}$ are
\[
(r, \bar{x})=(0,0), \quad E_{*}: (r, \bar{x})=(0,-m^{-1}\varepsilon_{0}SV^{2}),
\]
and the linearized matrices at these equilibria are
\begin{align*}
&(0,0):\left(\begin{array}{cc}
-(4m)^{-1}\varepsilon_{0}SV^{2} & 0 \\
0 & m^{-1}\varepsilon_{0}SV^{2}
\end{array}\right),
\\
&(0,-m^{-1}\varepsilon_{0}SV^{2}):\left(\begin{array}{cc}
-(4m)^{-1}\varepsilon_{0}SV^{2} & 0 \\
0 & -m^{-1}\varepsilon_{0}SV^{2}
\end{array}\right).
\end{align*}
Therefore, the equilibrium $(0,0)$ is a saddle, and $E_{*}: (0,-m^{-1}\varepsilon_{0}SV^{2})$ is sink.

For the dynamical system on the local chart $\{\bar{x}=\pm1\}$, we only need to consider the equilibria and their local stability at $\{r=0, \bar{\lambda}\ge 0\}$.
In conclusion, the origin on $\{\bar{x}=\pm1\}$ is a saddle, as in subsection \ref{sub:IKY-d4-2}.
The above discussion yields a qualitative flow schematic diagram similar to that shown in Figure \ref{fig:IKYsh-d2}.

\subsection{Dynamics and connecting orbits on the Poincar\'e-type disk}
\label{sub:IKYsh-dib3}
As in Subsection \ref{sub:IKYsh-d5}, by combining $\overline{U}_{j}$ ($j=1,2$) and $\overline{V}_{2}$ yields the dynamics on $\Phi_{\ge 0}$.
See also Figure \ref{fig:IKYsh-d4}.

\begin{figure}[t]
\centering
\includegraphics[scale=0.3]{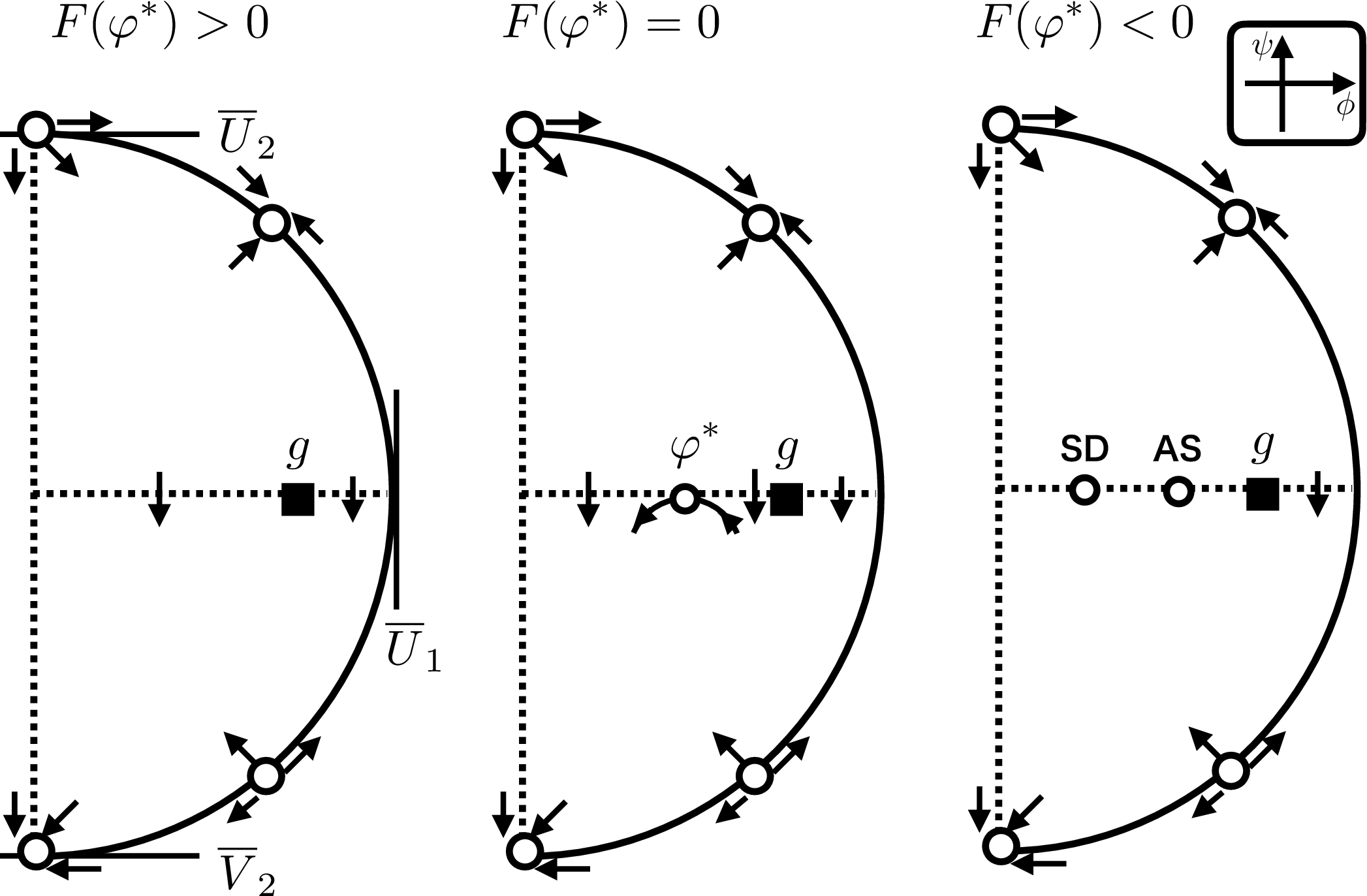}
\caption{
A schematic picture of the dynamics on a Poincar\'e-type disk for $b<0$. The
white circles represent equilibria, and the black squares represent the initial conditions imposed on \eqref{eq:IKYsh-f2} or \eqref{eq:IKYsh-int8}.}
\label{fig:IKYsh-d4}
\end{figure}

In this subsection, we discuss orbits that start from the initial conditions $(\phi, \psi)=(g,0)$.
The conclusion is the same for both $b>0$ and $b<0$, and is qualitatively identical to that in subsection \ref{sub:IKYsh-d5}.
When $F(\varphi^*) < 0$, the system converges to an asymptotically stable equilibrium on the $\phi$-axis.
When $F(\varphi^*) = 0$, the orbit goes to $(\phi, \psi)=(\varphi^{*}, 0)$ the one must approach $(\phi, \psi)=(0, -\infty)$ if $F(\varphi^{*})>0$.
Note that when $b > 0$, then $\varphi^{*} = \phi^{*}$, and when $b < 0$, then $\varphi^{*} = \phi_{*}$.
These conclusions are shown in the same manner as in subsection \ref{sub:IKYsh-d5}.

\section{Proof of theorems}
\label{sec:IKYsh-pro}
\subsection{Proof of Theorem \ref{thm:IKYsh-mr1}}
\label{sub:IKYsh-pro1}
\begin{proof}
We need to show that there exists a $|t_{+}| < +\infty$, and to derive the asymptotic forms in (i) and (ii).
In (iii), the fact that the value $x^{*}$ satisfies $0 < x^{*} < 3^{-1}g$ corresponds to the orbit of \eqref{eq:IKYsh-int8} of $b = 0$ converging to the equilibrium $(\phi, \psi) = (\phi_{2}, 0)$.

First, we will prove (i), that is the case for $K>0$.
The solution near the equilibrium $E^{*}$ in the blow-up vector field of $\{\bar{\lambda}=1\}$ with respect to the origin in the local chart $\overline{V}_{2}$ is approximated as follows:
\[
\begin{cases}
r(\sigma) = C_{1}e^{-(2m)^{-1}\varepsilon_{0}SV^{2}\sigma}(1+o(1)), \\
\bar{x}(\sigma) = C_{2}e^{-(2m)^{-1}\varepsilon_{0}SV^{2}\sigma}(1+o(1)) - \dfrac{\varepsilon_{0}SV^{2}}{m}
\end{cases}
\quad {\rm{as}} \quad \sigma\to +\infty
\]
with constants $C_{1, 2}>0$.
Using the time-rescaling and above relations, 
\begin{align*}
\dfrac{d\sigma}{d\xi}
&= \dfrac{d\sigma}{d\tau}\dfrac{d\tau}{ds}\dfrac{ds}{dt} 
= r^{3}\cdot \lambda^{-2}\cdot \phi^{-2} 
= r^{-3}\bar{x}^{-2} \\
&= \left\{ C_{1}e^{-(2m)^{-1}\varepsilon_{0}SV^{2}\sigma}(1+o(1)) \right\}^{-3} \cdot \left\{ C_{2}e^{-(2m)^{-1}\varepsilon_{0}SV^{2}\sigma}(1+o(1)) - \dfrac{\varepsilon_{0}SV^{2}}{m} \right\}^{-2} \\
&\sim C_{3}e^{\frac{3}{2}\frac{\varepsilon_{0}SV^{2}}{m}\sigma} 
\quad {\rm{as}} \quad \sigma\to +\infty
\end{align*}
holds.
Let 
\[
t_{+}=\lim_{\sigma \to +\infty}t(\sigma),
\]
then we have 
\[
t_{+} =\int_{0}^{+\infty} C_{4}e^{-\frac{3}{2}\frac{\varepsilon_{0}SV^{2}}{m}\sigma}\, d\sigma <+\infty.
\]
Therefore, we obtain the following asymptotic form:
\[
t_{+} -t \sim C_{5}e^{-\frac{3}{2}\frac{\varepsilon_{0}SV^{2}}{m}\sigma}
\quad {\rm{as}} \quad \sigma\to +\infty
\]
with a constant $C_{5}>0$.
Then, the asymptotic form of $x(t)$ is 
\begin{align*}
x(t)
&= g-\phi(t) \\
&= g+\lambda^{-1}x = g+r^{2}\bar{x} \\
&= g+ \left\{ C_{1}e^{-(2m)^{-1}\varepsilon_{0}SV^{2}\sigma}(1+o(1)) \right\}^{2} \cdot \left\{ C_{2}e^{-(2m)^{-1}\varepsilon_{0}SV^{2}\sigma}(1+o(1)) - \dfrac{\varepsilon_{0}SV^{2}}{m} \right\} \\
&\sim g - C_{6}e^{-\frac{\varepsilon_{0}SV^{2}}{m}\sigma} \\
&\sim g - C_{6}(t_{+}-t)^{\frac{2}{3}} 
\quad {\rm{as}} \quad t\to t_{+}-0
\end{align*}
with a constant $C_{6}>0$.
Hence, we see that \eqref{eq:IKYsh-mr1} holds.

Next, we will consider the case for $K=0$.
From \eqref{eq:IKYsh-b04}, the approximation of the solution on the center manifold at the equilibrium $E_{0}$ yields 
\[
\phi(s) =\dfrac{2}{3}g - \dfrac{c}{-kg s+C_{7}} +o(s^{-1})
\quad {\rm{as}} \quad s\to +\infty
\]
with a constant $C_{7}$.
Using the time-rescaling transformation, we have
\[
\dfrac{ds}{dt}=\phi^{-2} \sim \left(\dfrac{2}{3}g - \dfrac{c}{-kg s+C_{7}}\right)^{-2}
\quad {\rm{as}} \quad s\to +\infty
\]
and 
\begin{align*}
t+C_{8} 
&\sim \dfrac{c^{2}}{-k^{2}g^{2}s+C_{7}kg}+\dfrac{4c}{3k}\log(-kgs+C_{7})-\dfrac{4g}{9k}(-kgs+C_{7}) \\
&\sim \dfrac{4}{9}g^{2}s 
\quad {\rm{as}} \quad s\to +\infty
\end{align*}
with constants $C_{j}$.
From the above, the asymptotic behavior in (ii) has been derived.
Thus, this completes the proof of Theorem \ref{thm:IKYsh-mr1}.
\end{proof}

\subsection{Proof of Theorem \ref{thm:IKYsh-mr2}}
\label{sub:IKYsh-pro2}
\begin{proof}
The equations \eqref{eq:IKYsh-mr3} and \eqref{eq:IKYsh-mr4} characterize the behavior of solutions and are derived from the argument about the existence of finite equilibria in Section \ref{sec:IKYsh-f}.
The real root $\varphi^{*} \in (0, g)$ is given by $b>0$ in Proposition \ref{prop:IKYsh-f1} as $\varphi^{*}=\phi^{*}$, and for the case $b<0$ in Proposition \ref{prop:IKYsh-bf1} as $\varphi^{*}=\phi_{*}$.
\begin{enumerate}
\item[(i)]
When $F(\varphi^{*}) > 0$, we need to show that there exists $|t_{+}|<+\infty$ and that the asymptotic form is given by $\eqref{eq:IKYsh-mr1}$.
This is similar to Subsection $\eqref{sub:IKYsh-pro1}$.
Using the relationship between approximation of solutions near the equilibrium $E_{*}$ in the blow-up vector field on $\{\bar{\lambda}=1\}$ with respect to the origin in the local chart $\overline{V}_{2}$, we obtain 
\begin{align*}
\dfrac{d\sigma}{dt}
= r^{5}\cdot \lambda^{-3}\cdot \phi^{-2} 
= r^{-6}\bar{x}^{-2} 
\sim C_{3}e^{\frac{3}{2}\frac{\varepsilon_{0}SV^{2}}{m}\sigma} 
\quad {\rm{as}} \quad \sigma\to +\infty.
\end{align*}
This asymptotic form is the same as the one obtained in Subsection \ref{sub:IKYsh-pro1}.
By repeating the same argument as in Subsection \ref{sub:IKYsh-pro1}, we obtain 
\begin{align*}
x(t)
&= g-\phi(t) = g+r^{4}\bar{x} \\
%&= g+ \left\{ C_{10}e^{-(4m)^{-1}\varepsilon_{0}SV^{2}\sigma}(1+o(1)) \right\}^{4} \cdot \left\{ C_{11}e^{-m^{-1}\varepsilon_{0}SV^{2}\sigma}(1+o(1)) - m^{-1}\varepsilon_{0}SV^{2} \right\} \\
&\sim g - C_{12}e^{-\frac{\varepsilon_{0}SV^{2}}{m}\sigma} \\
&\sim g - C_{12}(t_{+}-t)^{\frac{2}{3}} 
\quad {\rm{as}} \quad t\to t_{+}-0
\end{align*}
with constants $C_{j}>0$.
Hence, Theorem \ref{thm:IKYsh-mr2} (i) has been proven.
\item[(ii)]
When $F(\varphi^{*})=0$, the evaluation of $x^{*}=g-\varphi^{*}$ are derived from Propositions \ref{prop:IKYsh-f1} and \ref{prop:IKYsh-bf1}.
The derivation of the asymptotic behavior follows a discussion similar to that in Subsection \ref{sub:IKYsh-pro1}.
From \eqref{eq:IKYsh-sp5}, the approximation of the solution on the center manifold near the equilibrium $(\phi, \psi)=(\varphi^{*}, 0)$ is
\[
\phi(s) =\varphi^{*} - \dfrac{2c}{-H(\varphi^{*})s+C_{13}} +o(s^{-1})
\quad {\rm{as}} \quad s\to +\infty
\]
with a constant $C_{13}$.
Using the time-rescaling, we have
\[
\dfrac{ds}{dt}=\phi^{-2} \sim \left(\varphi^{*} - \dfrac{2c}{-H(\varphi^{*}) s+C_{13}}\right)^{-2}
\quad {\rm{as}} \quad s\to +\infty
\]
and 
\begin{align*}
t+C_{14} 
&\sim \dfrac{4c^{2}}{-H(\varphi^{*})^{2}s+C_{13}H(\varphi^{*})}+\dfrac{4c}{H(\varphi^{*})}\log(-H(\varphi^{*})s+C_{13}) \\
&\quad\quad\quad -\dfrac{(\varphi^{*})^{2}}{H(\varphi^{*})}(-H(\varphi^{*})s+C_{13}) \\
&\sim (\varphi^{*})^{2}s 
\quad {\rm{as}} \quad s\to +\infty
\end{align*}
with the constants $C_{j}$.
Here, $s\to \infty$ corresponds to $t\to +\infty$.
Then, the asymptotic form of $\phi(t)$ is as follows:
\[
\phi(t)=\varphi^{*}+O(t^{-1}) \quad {\rm{as}} \quad t\to +\infty.
\]
\item[(iii)]
When $F(\varphi*) < 0$, it follows from Proposition \ref{prop:IKYsh-f1} (resp. Proposition \ref{prop:IKYsh-bf1}) that $\varphi^{*} < \phi_2$ (resp. $\varphi^{*} < \tilde{\phi}_2$).
This yields the result in (iii).
\end{enumerate}
Thus, this completes the proof of Theorem \ref{thm:IKYsh-mr2}.
\end{proof}

\section{Discussion}
\label{sec:IKYsh-c}
This paper focuses on a parallel-plate electrostatic actuator, the simplest representation of a MEMS structure.
We report on the differences in three characteristic behaviors (pull-in, the 1/3 rule, and touchdown) when modeled using linear, hard, and soft springs.

When $b=0$, the model are governed entirely by the linear restoration force and the electrostatic attraction.
According to Theorem \ref{thm:IKYsh-mr1}, the three fundamental behavior essential for MEMS development emerge naturally from the phase space geometry. By directly analyzing the governing differential equations rather than relying on standard nondimensionalization, we have provided an explicit mathematical characterization of these phenomena while preserving the clear physical identity of each control parameter. 
Furthermore, we clarified the asymptotic behavior corresponding to the speed at touchdown.
We revealed a mathematical structure in which pull-in corresponds to the bifurcation point of a finite equilibrium and originates from degeneracy.
In addition, the touchdown corresponds to a quench, which is a finite-time singularity phenomenon of the solution.
This can be interpreted as an orbit approaching an equilibrium at infinity.

In other words, Theorem \ref{thm:IKYsh-mr1} uses the pull-in voltage $V=V^{*}$ given by \eqref{eq:IKYsh-mr2} to show that as the voltage increases,
\begin{itemize}
\item For lower operating voltages satisfying $0 < V < V^*$, the condition $K < 0$ holds, corresponding to situation (iii), where a stable state exists.
\item At the (critical) pull-in voltage $V=V^{*}$, $K=0$ holds in situation (ii),
\item When the voltage exceeds the pull-in voltage $V>V^{*}$, a touchdown case (i) occurs and $K>0$ holds.
\end{itemize}
This rigorous classification corroborates well-known empirical facts in the conventional MEMS field from a solid, unified mathematical perspective.

Importantly, even when the cubic spring nonlinearity $b \neq 0$ is introduced, this global mathematical structure remains qualitatively and topologically invariant. 
That is, apart from the three known classifications established in Theorem \ref{thm:IKYsh-mr2}, the structure remains the same even when $b$ is introduced.
In terms of mathematical structure, pull-in behavior corresponds to the degeneracy and bifurcation point of finite equilibrium in the derived two-dimensional system.
In addition, the touchdown behavior corresponds to the finite-time singularity of solutions and can be obtained by extracting the dynamical system at infinity in phase space.

These mathematical results establish a structural framework that serves as a guiding principle for investigating the material properties of MEMS and evaluating their performance.
It represents the fundamental mathematical characteristics underlying the understanding of the structures of more complex modern MEMS devices.
In general, analyzing the degeneracy and finite-time singularity of solutions to nonlinear differential equations is a challenging mathematical problem, and these factors contribute to the abundance of solutions.
While our primary motivation stems from the need to predict and control instabilities in modern MEMS engineering, the intricate geometric structure of the phase space near infinity is inherently compelling from a purely mathematical interest.

If $b \neq 0$, then the pull-in voltage $V = V^*$ cannot be expressed explicitly as in \eqref{eq:IKYsh-mr2}.
That is, it is expressed implicitly as in \eqref{eq:IKYsh-mr3} and \eqref{eq:IKYsh-mr4}.
This implies the existence of the pull-in voltage $V = V^*$, and we list the consequences that follow from Theorem \ref{thm:IKYsh-mr2}.
\begin{itemize}
\item When a voltage $V$ is applied, the behavior of the two plates is shown as in (iii).
\item The system is in a pull-in state when the voltage is $V = V^*$ and $F(\varphi^*) = 0$.
\item Furthermore, a touchdown occur when a voltage slightly higher than the pull-in voltage is applied.
\end{itemize}
As shown in the results, we can see that the asymptotic forms and the touchdown time $t_{+}$ are independent of $b$.
This universal mathematical conclusion demonstrates that $b$ has no effect on touchdown.
From a MEMS safety perspective, this suggests that (new and more) additional experiments involving changes in spring characteristics corresponding to multiple parameter variations are unnecessary for preventing touchdown behavior in parallel-plate electrostatic actuators.

We also discuss the pull-in phenomenon.
Theorem \ref{thm:IKYsh-mr2} (ii) rigorously demonstrates that when a hard spring ($b>0$) is incorporated, the maximum stable displacement of the plate in the pull-in state can safely exceed the classical linear limit of $1/3$ of the initial gap $g$. Conversely, incorporating a soft spring ($b<0$) restricts the stable travel distance to strictly less than $g/3$.

This study provides mathematically rigorous results regarding three characteristic changes in behavior of a parallel-plate electrostatic actuator model resulting from the incorporation of a nonlinear spring. 
This problem originates from the MEMS field.
We demonstrate how an engineering problem can be successfully translated into mathematical terms and resolved using qualitative geometric methods.
As feedback to the MEMS field, Theorem \ref{thm:IKYsh-mr2} (ii) provides a design guideline stating that when manufacturing MEMS devices that displace by more than 1/3 of the initial gap, it is preferable to incorporate a hard spring with $b>0$.
The results presented in this paper concerning the relationship between the spring parameter $b$ and the displacement and behavior during pull-in and touchdown will not only lead to improved MEMS functionality but also contribute to the development of high-precision and safe MEMS design tools.

\section*{Acknowledgments}
This work was partly supported by the JSPS KAKENHI (Grant No. 22H01929, 25K17306), the JST A-STEP (Grant No. JPMJTR22R5), JST Program for co-creating startup ecosystem (Grant No. JPMJSF2315) and the grants from the Mitutoyo Association for Science and Technology, and Ritsumeikan Advanced Research Academy (RARA), Ritsumeikan University, Japan.

%%%%%%%%%%%%%%%%%%%%%%%%%%%%%%%%%%%%%%%%%%%%%%%%%%

\end{document}